\documentclass[envcountsame]{llncs} 
\usepackage{latexsym} 
\usepackage{times}

\usepackage[dvips]{epsfig} 
\graphicspath{{./Fig_eps/}{./Eps/}} 
\DeclareGraphicsExtensions{.ps,.eps} 
\usepackage{color} 
\usepackage{a4wide}

\title{The Order of the Giant Component of Random Hypergraphs}
 
\author{
Michael Behrisch\thanks{Supported by the DFG research center \textsc{Matheon} in Berlin.} \and
Amin Coja-Oghlan\thanks{Supported by the Deutsche Forschungsgemeinschaft (DFG FOR 413/1-2)} 
    \and 
    Mihyun Kang} 
\date{\today} 
\institute{Humboldt-Universit\"{a}t zu Berlin, Institut f\"{u}r Informatik,\\Unter den Linden 6, 10099 Berlin, Germany 
\\\email{coja@informatik.hu-berlin.de}}

\newcommand\eps{\varepsilon} 

\newcommand\ZZ{\bbbz} 
\newcommand\Var{\mathrm{Var}} 
\newcommand\Erw{\mathrm{E}} 
\newcommand\pr{\mathrm{P}} 

\newcommand\cA{\mathcal{A}}
\newcommand\cB{\mathcal{B}}
\newcommand\cD{\mathcal{D}}
\newcommand\cE{\mathcal{E}}
\newcommand\cH{\mathcal{H}}
\newcommand\cI{\mathcal{I}} 
\newcommand\cJ{\mathcal{J}} 
\newcommand\cK{\mathcal{K}} 
\newcommand\cL{\mathcal{L}} 
\newcommand\cS{\mathcal{S}} 
\newcommand\cF{\mathcal{F}} 
\newcommand\cQ{\mathcal{Q}} 
\newcommand\cW{\mathcal{W}} 
\newcommand\cP{\mathcal{P}} 
 
\newcommand\gnp{G(n,p)} 
 
\newcommand\hnm{\hdnm} 
\newcommand\hnp{\hdnp}

\newcommand\hdnm{H_d(n,m)} 
\newcommand\hdnp{H_d(n,p)} 

\newcommand\mybig{\mathrm{big}}
\newcommand\myiso{\mathrm{iso}}
 
\newcommand{\Bin}{{\rm Bin}} 
\newcommand{\eqref}[1]{(\ref{#1})} 
 
\newcommand\comp{\mathcal{C}}

\newcommand{\bink}[2] 
    {{{#1}\choose {#2}}} 
 
\newcommand\bc[1]{\left({#1}\right)}

\newcommand\brk[1]{\left\lbrack{#1}\right\rbrack}

\newcommand\abs[1]{\left|{#1}\right|}

\newcommand\ug[1]{\left\lceil{#1}\right\rceil} 
\newcommand\RR{\mathbf{R}} 
\newcommand\RRpos{\RR_{\geq0}} 
\newcommand\tOl{O(\mathrm{polylog}\,n)} 
\newcommand\tO[1]{O(#1\cdot\mathrm{polylog}\,n)} 
 
\newcommand{\whp}{w.h.p.} 
\newcommand{\stacksign}[2]{{\stackrel{\mbox{\scriptsize #1}}{#2}}} 

\newcommand\order{\mathcal{N}}

\newcommand{\Karonski}{Karo\'nski}
\newcommand{\Rucinski}{Ruci\'nski}
\newcommand{\Erdos}{Erd\H{o}s}
\newcommand{\Renyi}{R\'enyi}

\newcommand{\Luczak}{\L uczak}

\newcommand\Lem{Lemma}
\newcommand\Prop{Proposition}
\newcommand\Thm{Theorem}
\newcommand\Cor{Corollary}
\newcommand\Sec{Section}

\let\subs=\subseteq

\newcommand{\ex}[1]{{\mathrm{E}\left[#1\right]}}

\newcommand{\var}[1]{{{\rm Var}\left[#1\right]}}

\newcommand{\binnd}{{\bink{n-1}{d-1}}}

\newcommand{\Ya}{Y_\alpha}
\newcommand{\mua}{\mu_\alpha}
\newcommand{\Xa}{X_\alpha}

\newcommand{\Za}{Z_\alpha}
\newcommand{\Zab}{Z_{\alpha\beta}}
\newcommand{\Vab}{V_{\alpha\beta}}
\newcommand{\Yb}{Y_\beta}
\newcommand{\Yc}{Y_\gamma}
\newcommand{\mub}{\mu_\beta}
\newcommand{\muc}{\mu_\gamma}
\newcommand{\sumall}[1]{\sum_{#1\in\cA}}
\newcommand{\sumabc}{\sum_{\beta:\alpha\cap\beta\ne\emptyset}}
\newcommand{\sumabd}{\sum_{\beta:\alpha\cap\beta=\emptyset}}
\newcommand{\sumacc}{\sum_{\gamma:\alpha\cap\gamma\ne\emptyset}}
\newcommand{\sumacd}{\sum_{\gamma:\alpha\cap\gamma=\emptyset}}

\begin{document} 
 
\maketitle 

\begin{abstract}
We establish central and local limit theorems for the number of vertices in the largest component of a random
$d$-uniform hypergraph $\hnp$ with edge probability $p=c/\binnd$, where $(d-1)^{-1}+\eps<c<\infty$.
The proof relies on a new, purely probabilistic approach, and is based on Stein's method as well as
exposing the edges of $H_d(n,p)$ in several rounds.
\end{abstract}

\section{Introduction and Results}\label{Sec_Results} 


A \emph{$d$-uniform hypergraph} $H=(V,E)$ consists of a set $V$ of vertices 
and a set $E$ of edges, which are subsets of $V$ of cardinality $d$.
Moreover, a vertex $w$ is \emph{reachable in $H$} from a vertex $v$ if either 
$v=w$ or there is a sequence 
$e_1,\ldots,e_k$ of edges such that $v\in e_1$, $w\in e_k$, and 
$e_i\cap e_{i+1}\not=\emptyset$ for $i=1,\ldots,k-1$. 
Of course, reachability  in $H$ is an equivalence relation.
The equivalence classes are the \emph{components} of $H$, and
$H$ is \emph{connected} if there is only one component.

Throughout the paper, we let $V=\{1,\ldots,n\}$ be a set of $n$ vertices.
Moreover, if $2\leq d$ is a fixed integer and $0\leq p=p(n)\leq 1$ is sequence,
then we let $H_d(n,p)$ signify a random $d$-uniform hypergraph with vertex set $V$ in which each of the $\bink{n}d$
possible edges is present with probability $p$ independently.
We say that $H_d(n,p)$ enjoys some property $\cP$ \emph{with high probability} (\whp) if the probability that
$H_d(n,p)$ has $\cP$ tends to $1$ as $n\rightarrow\infty$.
If $d=2$, then the $H_d(n,p)$ model is identical with the well-known $G(n,p)$ model of random graphs.
In order to state some related results we will also need a different model  $\hnm$ of random hypergraphs,
where the hypergraph is chosen uniformly at random among all $d$-uniform hypergraphs with $n$
vertices and $m$ edges. 

\medskip

Since the pioneering work of \Erdos\ and \Renyi~\cite{ER60}, the component structure of random discrete structures has been a central theme in
probabilistic combinatorics.
In the present paper, we contribute to this theme by analyzing the maximum order $\order(\hnp)$ of a component of $\hnp$ in greater detail.
More precisely,  establishing central and local limit theorems for $\order(\hnp)$, we determine the asymptotic distribution of $\order(\hnp)$ precisely.
Though such limit theorems are known in the case of graphs (i.e, $d=2$), they are new in the case of $d$-uniform hypergraphs for $d>2$.
Indeed, to the best of our knowledge none of the arguments known for the graph case extends directly to the case of hypergraphs ($d>2$).
Therefore, we present a new, purely probabilistic proof of the central and local limit theorems, which, in contrast to most prior work,
does not rely on involved enumerative techniques.
We believe that this new technique is interesting in its own right and may have further applications.

\subsubsection{The giant component.}
In their seminal paper~\cite{ER60}, \Erdos\ and \Renyi\ proved that the number of vertices in the largest
component of $G(n,p)$ undergoes a \emph{phase transition} as $np\sim 1$.
They showed that if $np<1-\eps$ for an arbitrarily small $\eps>0$ that remains fixed as $n\rightarrow\infty$,
then all components of $G(n,p)$ consist of $O(\ln n)$ vertices.
By contrast, if $np>1+\eps$, then $G(n,p)$ has one \emph{giant} component on a linear number $\Omega(n)$ of vertices,
while all other components contain only $O(\ln n)$ vertices.
In fact, in the case $1+\eps<c=(n-1)p=O(1)$ \Erdos\ and \Renyi\ estimated the order (i.e., the number of vertices)
of the giant component:
let $\order(G(n,p))$ signify the maximum order of a component of $G(n,p)$.
Then
	\begin{equation}\label{eqIntro1}
	\mbox{$n^{-1}\order(G(n,p))$ converges in distribution to the constant $1-\rho$,}
	\end{equation}
where $0<\rho<1$ is the unique solution to the transcendental equation $\rho=\exp(c(\rho-1))$.

A corresponding result was established by Schmidt-Pruzan and Shamir~\cite{SPS85}
for random hypergraphs $\hnp$.
They showed that a random hypergraph $\hnp$ consists of components of order $O(\ln n)$ if $(d-1)\bink{n-1}{d-1}p<1-\eps$,
whereas $\hnp$ has a unique large (the \emph{giant}) component on $\Omega(n)$ vertices \whp\ if $(d-1)\bink{n-1}{d-1}p>1+\eps$.
Furthermore, Coja-Oghlan, Moore, and Sanwalani~\cite{CMS04} established a result similar to~(\ref{eqIntro1}), showing that 
in the case $(d-1)\bink{n-1}{d-1}p>1+\eps$ the
order of the giant component is $(1-\rho) n+o(n)$ \whp, where $0<\rho<1$ is the unique solution to the transcendental equation
	\begin{equation}\label{eqCOMV}
	\rho=\exp(c(\rho^{d-1}-1)).
	\end{equation}

\subsubsection{Central and local limit theorems.}

In terms of limit theorems, (\ref{eqIntro1}) provides a \emph{strong law of large numbers} for $\order(G(n,p))$,
i.e., it yields the probable value of $\order(G(n,p))$ up to fluctuations of order $o(n)$.
Thus, a natural question is whether we can characterize the distribution of $\order(G(n,p))$ (or $\order(\hnp)$) more precisely;
for instance, is it true that $\order(G(n,p))$ ``converges to the normal distribution'' in some sense?
Our first result, which we will prove in Section~\ref{Sec_Stein}, shows that this is indeed the case.

\begin{theorem}\label{Thm_NCLT}
Let $\cJ\subset((d-1)^{-1},\infty)$ be a compact interval, and let $0\leq p=p(n)\leq1$ be a sequence such that $c=c(n)=\bink{n-1}{d-1}p\in\cJ$ for all $n$.
Furthermore, let $0<\rho=\rho(n)<1$ be the unique solution to~(\ref{eqCOMV}), and set
	\begin{equation}\label{eqsigma}
	\sigma^2=\sigma(n)^2=\frac{\rho\brk{1-\rho+c(d-1)(\rho-\rho^{d-1})}n} 
			{(1-c(d-1)\rho^{d-1})^2}.
	\end{equation}
Then $\sigma^{-1}(\order(\hnp)-(1-\rho)n)$ converges in distribution to the standard normal distribution.
\end{theorem}
\Thm~\ref{Thm_NCLT} provides a \emph{central limit theorem} for $\order(\hnp)$;
it shows that for any fixed numbers $a<b$
	\begin{equation}\label{eqIntro2}
	\lim_{n\rightarrow\infty}\pr\brk{a\leq\frac{\order(\hnp)-(1-\rho)n}{\sigma}\leq b}=(2\pi)^{-\frac12}\int_a^b\exp(-t^2/2)dt
	\end{equation}
(provided that the sequence $p=p(n)$ satisfies the above assumptions).

Though \Thm~\ref{Thm_NCLT} provides quite useful information about the distribution of $\order(\hnp)$,
the main result of this paper is actually a \emph{local limit theorem} for $\order(\hnp)$, which
characterizes the distribution  of $\order(\hnp)$ even more precisely.
To motivate the local limit theorem,  we emphasize that \Thm~\ref{Thm_NCLT}
only estimates $\order(G(n,p))$ up to an error of $o(\sigma)=o(\sqrt{n})$.
That is, we do obtain from~(\ref{eqIntro2}) that for arbitrarily small but fixed $\gamma>0$
	\begin{equation}\label{eqIntro3}
	\pr\brk{|\order(\hnp)-\nu|\leq\gamma\sigma}\sim\frac{1}{\sqrt{2\pi}\sigma}
		\int_{-\gamma \sigma}^{\gamma\sigma}\exp\brk{\frac{(\nu-(1-\rho)n-t)^2}{2\sigma^2}}dt,
	\end{equation}
i.e., we can estimate the probability that $\order(\hnp)$ deviates from some value $\nu$ by at most $\gamma\sigma$.
However, it is impossible to derive from~(\ref{eqIntro2}) or~(\ref{eqIntro3}) the asymptotic probability that $\order(\hnp)$ \emph{hits $\nu$ exactly}.

By contrast, 
our next theorem 
shows that for any integer $\nu$ such that $|\nu-(1-\rho)n|\leq O(\sigma)$ we have
	\begin{equation}\label{eqIntro4}
	\pr\brk{\order(\hnp)=\nu}\sim\frac1{\sqrt{2\pi}\sigma}\exp\brk{-\frac{(\nu-(1-\rho)n)^2}{2\sigma^2}},
	\end{equation}
provided that $(d-1)^{-1}+\eps\leq\bink{n-1}{d-1}p=O(1)$.
Note that~(\ref{eqIntro4}) is exactly what we would obtain from~(\ref{eqIntro3}) if we were allowed to set
$\delta=\frac12\sigma(n,p)^{-1}$ in that equation.
Stated rigorously, the local limit theorem reads as follows.

\begin{theorem}\label{Thm_Nlocal}
Let $d\geq2$ be a fixed integer.
For any two compact intervals $\cI\subset\RR$, $\cJ\subset((d-1)^{-1},\infty)$,
and for any $\delta>0$ there exists $n_0>0$ such that the following holds.
Let $p=p(n)$ be  a sequence such that $c=c(n)=\bink{n-1}{d-1}p\in\cJ$ for all $n$,
let $0<\rho=\rho(n)<1$ be the unique solution to (\ref{eqCOMV}), and let $\sigma$ be as in~(\ref{eqsigma}).
If $n\geq n_0$ and if $\nu$ is an integer such that $\sigma^{-1}(\nu-(1-\rho)n)\in\cI$, then
	$$\frac{1-\delta}{\sqrt{2\pi}\sigma}\exp\brk{-\frac{(\nu-(1-\rho)n)^2}{2\sigma^2}}\leq\pr\brk{\order(\hnp)=\nu}\leq\frac{1+\delta}{\sqrt{2\pi}\sigma}\exp\brk{-\frac{(\nu-(1-\rho)n)^2}{2\sigma^2}}.$$
\end{theorem}

\subsubsection{Related work.}\label{Sec_related}

Since the work of \Erdos\ and \Renyi~\cite{ER60}, the component structure of $\gnp=H_2(n,p)$ has received considerable attention.
Stepanov~\cite{Ste70} provided central and local limit theorems for $\order(\gnp)$, thereby proving the $d=2$ case
of \Thm s~\ref{Thm_NCLT} and~\ref{Thm_Nlocal}.
In order to establish these limit theorems, he estimates the probability that a random graph $\gnp$ is connected up to a factor $1+o(1)$
using recurrence formulas for the number of connected graphs.
Furthermore,
Barraez, Boucheron, and Fernandez de la Vega~\cite{BBF00} reproved the central limit theorem for $\order(\gnp)$
via the analogy of breadth first search on a random graph and a Galton-Watson branching process.
In addition, a local limit theorem for $\order(\gnp)$ can also be derived using the
techniques of van der Hofstad and Spencer~\cite{HS05}, or the enumerative results of either Bender, Canfield, and McKay~\cite{BCM90}
or Pittel and Wormald~\cite{PW05}.

Moreover, Pittel~\cite{P90} proved a central limit theorem for the largest component in the $G(n,m)$ model of random graphs;
$G(n,m)$ is just a uniformly distributed graph with exactly $n$ vertices and $m$ edges.
Indeed, Pittel actually obtained his central limit theorem via a limit theorem for the joint distribution of the number of isolated trees
of a given order, cf.\ also Janson~\cite{JansonMST}.
A comprehensive treatment of further results on the components of $\gnp$ can be found in~\cite{JLR00}.

In contrast to the case of graphs, only little is known for $d$-uniform hypergraphs with $d>2$; for the methods used
for graphs do not extend to hypergraphs directly.
Using the result~\cite{KL97} on the number of sparsely connected hypergraphs,
\Karonski\ and \Luczak~\cite{KL02} investigated the phase transition of $\hnp$.
They established (among other things) a local limit theorem for $\order(\hnm)$ for
$m=n/d(d-1)+l$ and $1\ll \frac{l^3}{n^2} \le \frac{\ln n}{\ln\ln n}$
which is similar to $\hnp$ at the regime
$\bink{n-1}{d-1}p=(d-1)^{-1}+\omega$,
where $n^{-1/3}\ll\omega=\omega(n)\ll n^{-1/3}\ln n/\ln\ln n$.
These results were extended by Andriamampianina, Ravelomanana and Rijamamy~\cite{AR05,RR} to the regime $l = o(n^{1/3})$ 
($\omega=o(n^{-2/3})$ respectively).

By comparison, \Thm s~\ref{Thm_NCLT} and~\ref{Thm_Nlocal} deal with edge probabilities $p$ such that
$\bink{n-1}{d-1}p=(d-1)^{-1}+\Omega(1)$, i.e., $\bink{n-1}{d-1}p$ is bounded away from the critical point $(d-1)^{-1}$.
Thus, \Thm s~\ref{Thm_NCLT} and~\ref{Thm_Nlocal} complement~\cite{AR05,KL02,RR}.
The only prior paper dealing with $\bink{n-1}{d-1}p=(d-1)^{-1}+\Omega(1)$ is that of Coja-Oghlan, Moore, and Sanwalani~\cite{CMS04}, where the authors
computed the expectation and the variance of $\order(\hnp)$ and obtained qualitative results on the component
structure of $\hnp$.
In addition, in \cite{CMS04} the authors estimated the probability that $\hnp$ or a uniformly distributed $d$-uniform hypergraph $\hnm$
with $n$ vertices and $m$ edges is connected up to a constant factor.
While in the present work we build upon the results on the component structure of $\hnp$ from~\cite{CMS04},
the results and techniques of~\cite{CMS04} by themselves are not strong enough to obtain a central or even a local
limit theorem for $\order(\hnp)$.

\subsubsection{Techniques and outline.}\label{Sec_outline}

The aforementioned prior work~\cite{AR05,KL97,KL02} on the giant component for random hypergraphs
relies on enumerative techniques to a significant extent;
for the basis~\cite{AR05,KL97,KL02} are results on the asymptotic number of connected hypergraphs with a given number
of vertices and edges.
By contrast,  in the present work we employ neither enumerative techniques nor results,  but rely solely on probabilistic methods.
Our proof methods are also quite different from Stepanov's~\cite{Ste70},
who first estimates the asymptotic probability that a random graph $G(n,p)$ is connected in order to determine the distribution of $\order(\hnp)$.
By contrast, in the present work we prove the local limit theorem for $\order(\hnp)$ directly, thereby obtaining ``en passant''
a new proof for the local limit theorem for random graphs $G(n,p)$, which may be of independent interest.
Besides, the local limit theorem can be used to compute the asymptotic probability that $G(n,p)$ or, more generally, $\hnp$ is connected,
or to compute the asymptotic number of connected hypergraphs with a given number of vertices and edges (cf.~\Sec~\ref{Sec_Conclusion}).
Hence, the general approach taken in the present work is actually converse to the prior ones~\cite{AR05,KL97,KL02,Ste70}.

The proof of \Thm~\ref{Thm_NCLT} makes use of \emph{Stein's method}, which is a general technique for proving central limit theorems~\cite{Stein}.
Roughly speaking, Stein's result implies that a sum of a family of dependent random variables converges to the normal distribution if one can bound
the correlations within any constant-sized subfamily sufficiently well. 
The method was used by Barbour, \Karonski, and \Rucinski~\cite{BKR} in order to prove that in a random graph $G(n,p)$, e.g.,
the number of tree components of a given  (bounded) size is asymptotically normal.
To establish \Thm~\ref{Thm_NCLT}, we extend their techniques in two ways.
\begin{itemize}
\item Instead of dealing with the number of vertices in trees of a given size, we apply Stein's method to the \emph{total} number $n-\order(\hnp)$ of vertices 
	outside of the giant component;
	this essentially means that we need to sum over all possible (hyper)tree sizes up to about $\ln n$.
\item Since we are dealing with hypergraphs rather than graphs, we are facing a somewhat more complex situation than~\cite{BKR},
	because the fact that an edge may involve an arbitrary number $d$ of vertices yields additional dependencies.
\end{itemize}

The main contribution of this paper is the proof of \Thm~\ref{Thm_Nlocal}.
To establish this result, we think of the edges of $\hnp$ as being added in two ``portions''.
More precisely, we first include each possible edge with probability $p_1=(1-\eps)p$ independently, where $\eps>0$ is small but independent of $n$ (and denote the resulting random hypergraph by $H_1$);
by \Thm~\ref{Thm_NCLT}, the order $\order(H_1)$ of the largest component of $H_1$ is asymptotically normal.
Then, we add each possible edge that is not present in $H_1$ with a small probability $p_2\sim\eps p$ and investigate closely how
these additional random edges attach further vertices to the largest component of $H_1$.
Denoting the number of these ``attached'' vertices by $\cS$, we will show that the conditional distribution of $\cS$ \emph{given the
value of $\order(H_1)$} satisfies a local limit theorem.
Since $p_1$ and $p_2$ are chosen such that each edge is present with probability $p$ after the second portion of edges has been
added, this yields the desired result on $\order(\hnp)$.

The analysis of the conditional distribution of $\cS$ involves proving that $\cS$ is asymptotically normal.
To show this, we employ Stein's method once more.
In addition, in order to show that $\cS$ satisfies a \emph{local} limit theorem, we prove that the number of isolated vertices
of $H_1$ that get attached to the largest component of $H_1$ by the second portion of random edges is binomially distributed.
Since the binomial distribution satisfies a local limit theorem, we thus obtain a local limit theorem for $\cS$.

Our proof of \Thm~\ref{Thm_Nlocal} makes use of some results on the component structure of $\hnp$ derived in~\cite{CMS04}.
For instance, we employ the results on the expectation and the variance of $\order(\hnp)$ from that paper.
Furthermore, the analysis of $\cS$ given in the present work is a considerable extension of the argument used in~\cite{CMS04},
which by itself would just yield the probability that $\cS$ attains a specific value $s$ \emph{up to a constant factor}.

The main part of the paper is organized as follows.
After making some preliminaries in \Sec~\ref{Sec_Pre}, we outline the proof of \Thm~\ref{Thm_Nlocal} in \Sec~\ref{Sec_Nlocal}.
In that section we explain in detail how $\hnp$ is generated in two ``portions''.
Then, in \Sec~\ref{Sec_SSS} we analyze the random variable $\cS$, assuming the central limit theorem for $\cS$.
Further, \Sec~\ref{Sec_Stein} deals with the proof of \Thm~\ref{Thm_NCLT} and the proof of the central limit theorem for $\cS$ via Stein's method;
the reason why we defer the proof of \Thm~\ref{Thm_NCLT} to \Sec~\ref{Sec_Stein} is that we can use basically the same argument
to prove the asymptotic normality of both $\order(\hnp)$ and $\cS$.
Finally, \Sec~\ref{Sec_Conclusion} contains some concluding remarks, e.g., on the use of the present results to derive further limit theorems and to solve
enumerative problems.

\section{Preliminaries}\label{Sec_Pre}

Throughout the paper, we let $V=\{1,\ldots,n\}$.
If $d\geq2$ is an integer and  $V_1,\ldots,V_k\subset V$, then we let
$\cE_d(V_1,\ldots,V_k)$ signify the set of all subsets $e\subset V$ of cardinality $d$ such
that $e\cap V_i\not=\emptyset$ for all $i$.
We omit the subscript $d$ if it is clear from the context.

If $H$ is a hypergraph, then we let $V(H)$ denote its vertex set and $E(H)$ its edge set.
We say that a set $S\subset V(H)$ is \emph{reachable from} $T\subset V(H)$ if each vertex $s\in S$ is reachable from some vertex $t\in T$.
Further, if $V(H)\subset V=\{1,\ldots,n\}$, then the subsets of $V$ can be ordered lexicographically;
hence, we can define the \emph{largest component} of $H$ to be the lexicographically first component of order $\order(H)$.

We use the $O$-notation to express asymptotic estimates as $n\rightarrow\infty$.
Furthermore, if $f(x_1,\ldots,x_k,n)$ is a function that depends not only on $n$ but also on some further parameters $x_i$
from domains $D_i\subset\RR$ ($1\leq i\leq k$), and if $g(n)\geq0$ is another function, then we say that the estimate
$f(x_1,\ldots,x_k,n)=O(g(n))$ holds \emph{uniformly in $x_1,\ldots,x_k$} if the following is true:
if $\cI_j$ and $D_j$, $\cI_j\subset D_j$, are compact sets, then there exist numbers $C=C(\cI_1,\ldots,\cI_k)$ and $n_0=n_0(\cI_1,\ldots,\cI_k)$
such that $|f(x_1,\ldots,x_k,n)|\leq Cg(n)$ for all $n\geq n_0$ and $(x_1,\ldots,x_k)\in\prod_{j=1}^k\cI_j$.
We define uniformity analogously for the other Landau symbols $\Omega$, $\Theta$, etc.

We shall make repeated use of the following \emph{Chernoff bound} on the tails of a binomially distributed variable $X=\Bin(\nu,q)$
(cf.~\cite[p.~26]{JLR00} for a proof):
for any $t>0$ we have
	\begin{equation}\label{eqChernoff}
	\pr\brk{\abs{X-\Erw(X)}\geq t}\leq2\exp\bc{-\frac{t^2}{2(\Erw(X)+t/3)}}.
	\end{equation}
Moreover, we employ the following \emph{local limit theorem} for the binomial distribution (cf.~\cite[Chapter~1]{BB}).
\begin{proposition}\label{Prop_Bin}
Suppose that $0\leq p=p(n)\leq1$ is a sequence such that $np(1-p)\rightarrow\infty$ as $n\rightarrow\infty$.
Let $X=\Bin(n,p)$.
Then for any sequence $x=x(n)$ of integers such that $|x-np|=o(np(1-p))^{2/3}$,
	$$\pr\brk{X=x}\sim\bc{2\pi np(1-p)}^{-\frac12}\exp\bc{-\frac{(x-np)^2}{2p(1-p)n}}\qquad
		\mbox{as }n\rightarrow\infty.$$
\end{proposition}

Furthermore, we make use of the following theorem,
which summarizes results from~\cite[\Sec~6]{CMS04} on the component structure of $\hnp$.
\begin{theorem}\label{Thm_global} 
Let $p=c\bink{n-1}{d-1}^{-1}$.
\begin{enumerate}
\item If there is a fixed $c_0<(d-1)^{-1}$ such that $c=c(n)\leq c_0$, then
		$$\pr\brk{\order(\hnp)\leq3(d-1)^2(1-(d-1)c_0)^{-2}\ln n}\geq1-n^{-100}.$$
\item Suppose that $c_0>(d-1)^{-1}$ is a constant, and that $c_0\leq c=c(n)=o(\ln n)$ as $n\rightarrow\infty$.
	Then the transcendental equation~(\ref{eqCOMV}) has a unique solution $0<\rho=\rho(c)<1$,
	which satisfies
		\begin{equation}\label{eqrhoupper}
		\bink{\rho n}{d-1}p<c_0'<(d-1)^{-1}.
		\end{equation}
	for some number $c_0'>0$ that depends only on $c_0$.
	Moreover,
		\begin{eqnarray*}
		\abs{\Erw\brk{\order(\hnp)}-(1-\rho)n}&\leq&n^{o(1)},\label{eqglobalexp}\\
		\Var(\order(\hnp))&\sim&\frac{\rho\brk{1-\rho+c(d-1)(\rho-\rho^{d-1})}n} 
			{(1-c(d-1)\rho^{d-1})^2}.\label{eqglobalvar}
		\end{eqnarray*}
	Furthermore, with probability $\geq1-n^{-100}$ there is precisely one
	component of order $(1+o(1))(1-\rho)n$ in $\hnp$, while all other components have order $\leq\ln^2 n$.
	In addition,
		$$\pr\brk{|\order(\hnp)-\Erw(\order(\hnp))|\geq n^{0.51}}\leq n^{-100}.$$
\end{enumerate}
\end{theorem} 

Finally, the following result on the component structure of
$\hnp$ with average degree $\bink{n-1}{d-1}p<(d-1)^{-1}$ below the threshold has been derived in~\cite[\Sec~6]{CMS04}
via the theory of branching processes.

\begin{proposition}\label{Prop_branching}
There exists a function $q:(0,(d-1)^{-1})\times\brk{0,1}\rightarrow\RRpos$,
$(\zeta,\xi)\mapsto q(\zeta,\xi)=\sum_{k=1}^{\infty}q_k(\zeta)\xi^k$ whose coefficients $\zeta\mapsto q_k(\zeta)$ are
differentiable such that the following holds.
Suppose that $0\leq p=p(n)\leq1$ is a sequence such that $0<\bink{n-1}{d-1}p=c=c(n)<(d-1)^{-1}-\eps$
for an arbitrarily small $\eps>0$ that remains fixed as $n\rightarrow\infty$.
Let $P(c,k)$ denote the probability that in $\hnp$ some fixed vertex $v\in V$ lies in a component of order $k$.
Then
	$$P(c,k)=(1+o(n^{-2/3}))q_k(c)\qquad
		\mbox{for all }1\leq k\leq\ln^2n.$$
Furthermore, for any fixed $\eps>0$ there is a number $0<\gamma=\gamma(\eps)<1$ such that
	\begin{equation}\label{eqqkcbound}
	q_k(c)\leq\gamma^k\quad\mbox{for all }0<c<(d-1)^{-1}-\eps.
	\end{equation}
\end{proposition}

\section{Proof of \Thm~\ref{Thm_Nlocal}}\label{Sec_Nlocal}

\emph{Throughout this section, we assume that $c=c(n)=\bink{n-1}{d-1}p\in\cJ$ for some compact interval $\cJ\subset((d-1)^{-1},\infty)$.
Moreover, we let $\cI\subset\RR$ be some fixed compact interval, and $\nu$ denotes an integer such that
$(\nu-(1-\rho)n)/\sigma\in\cI$.
All asymptotics are understood to hold uniformly in $c$ and $(\nu-(1-\rho)n)/\sigma$.}

\subsection{Outline}\label{Sec_NlocalOutline}

Let $\eps=\eps(\cJ)>0$ be independent of $n$ and small enough so that $(1-\eps)\bink{n-1}{d-1}p>(d-1)^{-1}+\eps$.
Set $p_1=(1-\eps)p$.
Moreover, let $p_2$ be the solution to the equation $p_1+p_2-p_1p_2=p$; then $p_2\sim\eps p$.
We expose the edges of $\hnp$ in four ``rounds'' as follows.
\begin{description}
\item[R1.] As a first step, we let $H_1$ be a random hypergraph obtained by
	including each of the $\bink{n}d$ possible edges with probability $p_1$ independently.
	Let $G$ denote the largest component of $H_1$.
\item[R2.] Let $H_2$ be the hypergraph obtained from $H_1$ by adding each
	edge $e\not\in H_1$ that lies completely outside of $G$ (i.e., $e\subset V\setminus G$) with probability $p_2$ independently.
\item[R3.]
	Obtain $H_3$ by adding each possible edge $e\not\in H_1$ that contains vertices
	of both $G$ and $V\setminus G$ with probability $p_2$ independently.
\item[R4.]
	Finally, include each possible edge $e\not\in H_1$ such that $e\subset G$ with probability $p_2$ independently.
\end{description}
Here the 1st round corresponds to the first portion of edges mentioned in \Sec~\ref{Sec_outline},
and the edges added in the 2nd--4th round correspond to the second portion.
Note that for each possible edge $e\subset V$ the probability that $e$ is actually present in $H_4$ is $p_1+(1-p_1)p_2=p$,
hence $H_4=\hnp$.
Moreover, as $\bink{n-1}{d-1}p_1>(d-1)^{-1}+\eps$ by our choice of $\eps$,
\Thm~\ref{Thm_global} entails that \whp\ $H_1$ has exactly one largest component of linear size $\Omega(n)$ (the ``giant component'').
Further, the edges added in the 4th round do not affect the order of the largest component, i.e., 
$\order(H_4)=\order(H_3)$.

In order to analyze the distribution of $\order(\hnp)$, we first
establish \emph{central limit theorems} for $\order(H_1)=|G|$ and $\order(H_3)=\order(H_4)=\order(\hnp)$,
i.e., we prove that (centralized and normalized versions of) $\order(H_1)$ and $\order(H_3)$ are asymptotically normal.
Then, we investigate the number of vertices $\cS=\order(H_3)-\order(H_1)$ that get attached to $G_1$ during the 3rd round.
We shall prove that \emph{given that $|G|=n_1$}, $\cS$ is \emph{locally} normal with mean $\mu_{\cS}+(n_1-\mu_1)\lambda_\cS$
and variance $\sigma_{\cS}^2$ independent of $n_1$.
Finally, we combine these results to obtain the local limit theorem for $\order(\hnp)=\order(H_3)=\order(H_1)+\cS$.

Let $c_1=\bink{n-1}{d-1}p_1$ and $c_3=\bink{n-1}{d-1}p$.
Moreover, let $0<\rho_3<\rho_1<1$ signify the solutions to the transcendental equations $\rho_j=\exp\brk{c_j(\rho_j^{d-1}-1)}$
and set for $j=1,3$
	$$\mu_j=(1-\rho_j)n,\quad \sigma_j^2=\frac{\rho_j\brk{1-\rho_j+c_j(d-1)(\rho_j-\rho_j^{d-1})}n} 
			{(1-c_j(d-1)\rho_j^{d-1})^2}\quad\mbox{(cf.~\Thm~\ref{Thm_global})}.$$
The following proposition, which we will prove in \Sec~\ref{Sec_Stein}, establishes a central limit theorem for both $\order(H_1)$ and $\order(H_3)$ and thus proves Theorem~\ref{Thm_NCLT}.
\begin{proposition}\label{Prop_NCLT}
$(\order(H_j)-\mu_j)/\sigma_j$ converges in distribution to the standard normal distribution for $j=1,3$.
\end{proposition}

With respect to the distribution of $\cS$, we will establish the following local limit theorem in \Sec~\ref{Sec_SSS}.
\begin{proposition}\label{Prop_SSS}
Suppose that $|n_1-\mu_1|\leq n^{0.6}$.
\begin{enumerate}
\item The conditional expectation of $\cS$ given that $\abs{G}=n_1$ satisfies $\Erw(\cS|\order_1=n_1)=\mu_\cS+\lambda_\cS(n_1-\mu_1)+o(\sqrt{n})$,
	where $\mu_\cS=\Theta(n)$ and $\lambda_\cS=\Theta(1)$ are independent of $n_1$.
\item There is a constant $C>0$ such that for all $s$ satisfying
		$|\mu_\cS+\lambda_\cS(n_1-\mu_1)-s|\leq n^{0.6}$ we have
		$\pr\brk{\cS=\nu|\order_1=n_1}\leq Cn^{-\frac12}.$
\item If $s$ is an integer such that $|\mu_\cS+\lambda_\cS(n_1-\mu_1)-s|\leq O(\sqrt{n})$, then
		$$\pr\brk{\cS=s|\order_1=n_1}\sim\frac1{\sqrt{2\pi}\sigma_\cS}\exp\bc{-\frac{(\mu_\cS+\lambda_\cS(n_1-\mu_1)-s)^2}{2\sigma_\cS^2}},$$
	where $\sigma_\cS=\Theta(\sqrt{n})$ is independent of $n_1$.
\end{enumerate}
\end{proposition}
Since $\order_3=\order_1+\cS$, \Prop s~\ref{Prop_NCLT} and~\ref{Prop_SSS} yield
	\begin{equation}\label{eqmus}
	\mu_3=\mu_1+\mu_\cS+o(\sqrt{n}).
	\end{equation}

Combining \Prop s~\ref{Prop_NCLT} and~\ref{Prop_SSS}, we derive the following formula for $\pr\brk{\order_3=\nu}$
in \Sec~\ref{Sec_CorSSS}.
Recall that we are assuming that $\nu$ is an integer such that $(\nu-\mu)/\sigma=(\nu-\mu_3)/\sigma_3\in\cI$.

\begin{corollary}\label{Cor_SSS}
Letting $z=(\nu-\mu_3)/\sigma_3$, we have
	\begin{equation}\label{eqCorSSS}
	\pr\brk{\order_3=\nu}\sim\frac{1}{2\pi\sigma_\cS}\int_{-\infty}^{\infty}\exp\brk{-\frac{x^2}2-
			\frac12\bc{(x\cdot(1+\lambda_\cS)\frac{\sigma_1}{\sigma_\cS}-z\cdot\frac{\sigma_3}{\sigma_\cS}}^2}dx.
	\end{equation}
\end{corollary}

\noindent
\emph{Proof of \Thm~\ref{Thm_Nlocal}.}
Integrating the right hand side of~(\ref{eqCorSSS}),
we obtain an expression of the form
	\begin{equation}\label{eqNlocal1}
	\pr\brk{\order_3=\nu}\sim\frac1{\sqrt{2\pi}\tau}\exp\bc{-\frac{(\nu-\kappa)^2}{2\tau^2}},
	\end{equation}
where $\kappa,\tau^2=\Theta(n)$.
Therefore, on the one hand $(\order_3-\mu_3)/\sigma_3$ converges in distribution to the normal distribution with mean
$\kappa-\mu_3$ and variance $(\tau/\sigma_3)^2$.
On the other hand, \Prop~\ref{Prop_NCLT} states that $(\order_3-\mu_3)/\sigma_3$ converges to the standard normal distribution.
Consequently, $|\kappa-\mu_3|=o(\tau)$ and $\tau\sim\sigma_3$.
Plugging these estimates into (\ref{eqNlocal1}), we obtain
	$\pr\brk{\order_3=\nu}\sim\frac1{\sqrt{2\pi}\sigma_3}\exp\bc{-\frac12(\nu-\mu_3)^2\sigma_3^{-2}}$.
Since $\order_3=\order(\hnp)$, this yields the assertion.
\qed

\subsection{Proof of \Cor~\ref{Cor_SSS}}\label{Sec_CorSSS}

Let $\alpha>0$ be arbitrarily small but fixed as $n\rightarrow\infty$, and let $C'=C'(\alpha)>0$ be a large enough number depending only on $\alpha$.
Set $J=\{n_1\in\ZZ:|n_1-\mu_1|\leq C'\sqrt{n}\}$, let $J'=\{n_1\in\ZZ:C'\sqrt{n}<|n_1-\mu_1|\leq n^{0.6}\}$, and $J''=\{n_1\in\ZZ:|n_1-\mu_1|>n^{0.6}\}$.
Then letting
	$$\Psi_X=\sum_{n_1\in X}\pr\brk{\order_1=n_1}\pr\brk{\cS=\nu-n_1|\order_1=n_1},\textrm{\quad for } X\in\{J,J',J''\}$$
we have $\pr\brk{\order_3=\nu}=\Psi_J+\Psi_{J'}+\Psi_{J''}$,
and we shall estimate each of the three summands individually.

Since \Thm~\ref{Thm_global} implies that $\pr\brk{|\order_1-\mu_1|>n^{0.51}}\leq n^{-100}$, we conclude that
	\begin{equation}\label{eqPfNlocal1}
	\Psi_{J''}\leq\pr\brk{\order_1\in J''}\leq n^{-100}.
	\end{equation}
Furthermore, as $\sigma_1^2=O(n)$, Chebyshev's inequality implies that
	\begin{equation}\label{eqPfNlocal2}
	\pr\brk{\order_1\in J'}\leq\pr\brk{|\order_1-\mu_1|>C'\sqrt{n}}\leq\sigma_1^2 C'^{-2}n^{-1}<\alpha/C',
	\end{equation}
provided that $C'$ is large enough.
Hence, combining~(\ref{eqPfNlocal2}) with the second part of \Prop~\ref{Prop_SSS}, we obtain
	\begin{equation}\label{eqPfNlocal3}
	\Psi_{J'}\leq
		\pr\brk{\order_1\in J'}\cdot\frac{C}{\sqrt{n}}\leq\frac{\alpha C}{C'\sqrt{n}}<\alpha n^{-1/2},
	\end{equation}
where once we need to pick $C'$ sufficiently large.

To estimate the contribution of $n_1\in J$, we split $J$ into subintervals $J_1,\ldots,J_K$ of length
between $\frac{\sigma_1}{2C'}$ and $\frac{\sigma_1}{C'}$.
Moreover, let $I_j$ be the interval $\lbrack(\min J_j-\mu_1)/\sigma_1,(\max J_j-\mu_1)/\sigma_1\rbrack$.
Then \Prop~\ref{Prop_NCLT} implies that
	\begin{equation}\label{eqPfNlocal4}
	\frac{1-\alpha}{\sqrt{2\pi}}\int_{I_j}\exp(-x^2/2)dx\leq\sum_{n_1\in J_j}\pr\brk{\order_1=n_1}\leq\frac{1+\alpha}{\sqrt{2\pi}}
			\int_{I_j}\exp(-x^2/2)dx
	\end{equation}
for each $1\leq j\leq K$.
Furthermore, \Prop~\ref{Prop_SSS} yields
	$$\pr\brk{\cS=\nu-n_1|\order_1=n_1}\sim\frac1{\sqrt{2\pi}\sigma_{\cS}}\exp\bc{-\frac{(\nu-n_1-\mu_\cS-\lambda_\cS(n_1-\mu_1))^2}{2\sigma_\cS^2}}.$$
for each $n_1\in J$.
Hence, choosing $C'$ sufficiently large, we can achieve that
for all $n_1\in J_j$ and all $x\in I_j$ the bound
	\begin{eqnarray}\nonumber
	\pr\brk{\cS=\nu-n_1|\order_1=n_1}&\leq&
		\frac{(1+\alpha)^2}{\sqrt{2\pi}\sigma_{\cS}}\exp\bc{-\frac{(\nu-\mu_1-\sigma_1x-\mu_\cS-\lambda_\cS(n_1-\mu_1))^2}{2\sigma_\cS^2}}\\
		&\stacksign{(\ref{eqmus})}{\sim}&\;
					\frac{(1+\alpha)^2}{\sqrt{2\pi}\sigma_{\cS}}\exp\bc{-\frac12\bc{(x\cdot(1+\lambda_\cS)\frac{\sigma_1}{\sigma_\cS}-z\cdot\frac{\sigma_3}{\sigma_\cS}}^2}
	\label{eqPfNlocal6}
	\end{eqnarray}
holds.
Now, combining~(\ref{eqPfNlocal4}) and~(\ref{eqPfNlocal6}), we conclude that
	\begin{eqnarray}
	\Psi_J
		&=&\sum_{j=1}^K\sum_{n_1\in J_j}\pr\brk{\order_1=n_1}\pr\brk{\cS=\nu-n_1|\order_1=n_1}\nonumber\\
		&\leq&
			\frac{(1+\alpha)^3}{2\pi\sigma_\cS}\sum_{j=1}^K\int_{I_j}\exp\brk{-\frac{x^2}2-
			\frac12\bc{(x\cdot(1+\lambda_\cS)\frac{\sigma_1}{\sigma_\cS}-z\cdot\frac{\sigma_3}{\sigma_\cS}}^2}dx\nonumber\\
		&\leq&
			\frac{1+4\alpha}{2\pi\sigma_\cS}\int_{-\infty}^{\infty}\exp\brk{-\frac{x^2}2-
			\frac12\bc{(x\cdot(1+\lambda_\cS)\frac{\sigma_1}{\sigma_\cS}-z\cdot\frac{\sigma_3}{\sigma_\cS}}^2}dx.
		\label{eqPfNlocal7}
	\end{eqnarray}
Analogously, we derive the matching lower bound
	\begin{equation}\label{eqPfNlocal8}
	\Psi_J\geq\frac{1-4\alpha}{2\pi\sigma_\cS}\int_{-\infty}^{\infty}\exp\brk{-\frac{x^2}2-
			\frac12\bc{(x\cdot(1+\lambda_\cS)\frac{\sigma_1}{\sigma_\cS}-z\cdot\frac{\sigma_3}{\sigma_\cS}}^2}dx.
	\end{equation}
Finally, combining~(\ref{eqPfNlocal1}), (\ref{eqPfNlocal3}), (\ref{eqPfNlocal7}), and~(\ref{eqPfNlocal8}),
and remembering that $\pr\brk{\order_3=\nu}=\Psi_J+\Psi_{J'}+\Psi_{J''}$, we obtain the assertion,
because $\alpha>0$ can be chosen arbitrarily small if $n$ gets sufficiently large.

\section{The Conditional Distribution of $\cS$}\label{Sec_SSS}

\emph{Throughout this section, we keep the notation and the assumptions from \Sec~\ref{Sec_Nlocal}.
In addition, we let $G\subset V$ be a set of cardinality $n_1$ such that $|n_1-\mu_1|\leq n^{0.6}$.}

\subsection{Outline}

The goal of this section is to prove \Prop~\ref{Prop_SSS}.
Let us condition on the event that the largest component of $H_1$ is $G$.
To analyze the conditional distribution of $\cS$, we need to overcome the problem that in $H_1$ the edges in the set $V\setminus G$ do not occur independently
anymore once we condition on $G$ being the largest component of $H_1$.
However, we will see that this conditioning is ``not very strong''.
To this end, we shall compare $\cS$ with an ``artificial'' random variable $\cS_G$, which models the edges
contained in $V\setminus G$ as mutually independent objects.
To define $\cS_G$, we set up random hypergraphs $H_{j,G}$, $j=1,2,3$, in three ``rounds'' as follows.
\begin{description}
\item[R1'.] The vertex set of $H_{1,G}$ is $V=\{1,\ldots,n\}$, and each of the
	$\bink{n-n_1}d$ possible edges $e\subset V\setminus G$ is present in $H_{1,G}$ with probability $p_1$ independently.
\item[R2'.] Adding each possible edge $e\subset V\setminus G$ not present in $H_{1,G}$
	with probability $p_2$ independently yields $H_{2,G}$.
\item[R3'.] Obtain $H_{3,G}$ from $H_{2,G}$ by including each possible edge $e$ incident to
	both $G$ and $V\setminus G$ with probability $p_2$ independently.
\end{description}

The process R1'--R3' relates to the process R1--R4 from \Sec~\ref{Sec_NlocalOutline} as follows.
While in $H_1$ the edges in $V\setminus G$ are mutually dependent, we have ``artificially'' constructed $H_{1,G}$ in such a way
that the edges outside of $G$ occur independently.
Then, $H_{2,G}$ and $H_{3,G}$ are obtained similarly as $H_2$ and $H_3$,
namely by including further edges inside of $V\setminus G$ and crossing edges between $G$ and $V\setminus G$ with probability $p_2$.
Letting $S_G$ denote the set of vertices in $V\setminus G$ that are reachable from $G$,
the quantity $\cS_G=|S_G|$ now corresponds to $\cS$.
In contrast to R1--R4, the process R1'--R3' completely disregards edges inside of $G$, because these do not affect $\cS_G$.
The following lemma, which we will prove in \Sec~\ref{Sec_SSSGGG} shows that $\cS_G$ is indeed a very good approximation of $\cS$, so that it suffices to study $\cS_G$.
\begin{lemma}\label{Lemma_SSSGGG}
For any $\nu\in\ZZ$ we have
	$\abs{\pr\brk{\cS=\nu\;|\;\order(H_1)=n_1}-
		\pr\brk{\cS_G=\nu}}\leq n^{-9}.$
\end{lemma}

As a next step, we investigate the expectation of $\cS_G$.
While there is no need to compute $\Erw(\cS_G)$ precisely, we do need that $\Erw(\cS_G)$
depends on $n_1-\mu_1$ \emph{linearly}. The corresponding proof can be found in \Sec~\ref{Sec_ESSSGGG}.

\begin{lemma}\label{Lemma_ESSSGGG}
We have $\Erw(\cS_G)=\mu_\cS+\lambda_\cS(n_1-\mu_1)+o(\sqrt{n})$,
where $\mu_\cS=\Theta(n)$ and $\lambda_\cS=\Theta(1)$ do not depend on $n_1$.
\end{lemma}
Furthermore, we need that the variance of $\cS_G$ is essentially independent of the
precise value of $n_1$. This will be proven in \Sec~\ref{Sec_VarSSSGGG}.

\begin{lemma}\label{Lemma_VarSSSGGG}
We have $\Var(\cS_G)=O(n)$.
Moreover, if $G'\subset V$ is another set such that  $|\mu_1-|G'||=o(n)$,
then $|\Var(\cS_G)-\Var(\cS_{G'})|=o(n)$.
\end{lemma}

To show that $\cS_G$ satisfies a local limit theorem, the crucial step is to prove that for numbers
$s$ and $t$ such that $s$ is ``close''  to $t$ the probabilities $\pr\brk{\cS_G=s}$, $\pr\brk{\cS_G=t}$ are ``almost the same''.
More precisely, the following lemma, proven in \Sec~\ref{Sec_SSSGGGlocal}, holds.

\begin{lemma}\label{Lemma_SSSGGGlocal}
For every $\alpha>0$ there is $\beta>0$ such that for all $s,t$ satisfying
$|s-\Erw(\cS_G)|,|t-\Erw(\cS_G)|\leq n^{0.6}$ and
$|s-t|\leq\beta n^{1/2}$ we have
	$$(1-\alpha)\pr\brk{\cS_G=s}-n^{-10}
		\leq\pr\brk{\cS_G=t}\leq(1+\alpha)\pr\brk{\cS_G=s}+n^{-10}.$$
Moreover, there is a constant $C>0$ such that
	$\pr\brk{\cS_G=s}\leq Cn^{-1/2}$ for all integers $s$.
\end{lemma}
Letting $G_0=\{1,\ldots,\lceil\mu_1\rceil\}$, we define $\sigma_\cS^2=\Var(\cS_{G_0})$
and obtain a lower bound on $\sigma_\cS$ as an immediate consequence of \Lem~\ref{Lemma_SSSGGGlocal}.

\begin{corollary}\label{Cor_VarSSSGGG}
We have $\sigma_\cS=\Omega(\sqrt{n})$.
\end{corollary}
\begin{proof}
By \Lem~\ref{Lemma_SSSGGGlocal} there exists a number $0<\beta<0.01$ independent of $n$ such that
for all integers $s,t$ satisfying $|s-\Erw(\cS_G)|,|t-\Erw(\cS_G)|\leq\sqrt{n}$ and $|s-t|\leq\beta\sqrt{n}$ we have
	\begin{equation}\label{eqCorVarSSSGGG1}
	\pr\brk{\cS_G=t}\geq\frac23\pr\brk{\cS_G=s}-n^{-10}.
	\end{equation}
Set $\gamma=\beta^2/64$ and assume for contradiction that $\sigma_\cS^2<\gamma n/2$.
Moreover, suppose that $G=G_0=\{1,\ldots,\lceil\mu_1\rceil\}$.
Then Chebyshev's inequality entails that $\pr\brk{|\cS_G-\Erw(\cS_G)|\geq\sqrt{\gamma n}}\leq\frac12$.
Hence, there exists an integer $s$ such that $|s-\Erw(\cS_G)|\leq\sqrt{\gamma n}$ and $\pr\brk{\cS_G=s}\geq\frac12(\gamma n)^{-\frac12}$.
Therefore, due to~(\ref{eqCorVarSSSGGG1}) we have $\pr\brk{\cS_G=t}\geq\frac14(\gamma n)^{-\frac12}$ for all integers
$t$ such that $|s-t|\leq\beta\sqrt{n}$.
Thus, recalling that $\gamma=\beta^2/64$, we obtain
	$1\geq\pr\brk{|\cS_G-s|\leq\beta\sqrt{n}}=\sum_{t:|t-s|\leq\beta\sqrt{n}}\pr\brk{\cS_G=t}
			\geq\frac{\beta\sqrt{n}}{4\sqrt{\gamma n}}>1.$
This contradiction shows that $\sigma_\cS^2\geq\gamma n/2$.
\qed\end{proof}

Using the above estimates of the expectation and the variance of $\cS_G$ and invoking Stein's method
once more, in \Sec~\ref{Sec_Stein} we will
show the following.

\begin{lemma}\label{Lemma_SSSGGGNormal}
If $|n_1-\mu_1|\leq n^{0.66}$, then $(\cS_G-\Erw(\cS_G))/\sigma_\cS$ is asymptotically normal.
\end{lemma}

\noindent\emph{Proof of \Prop~\ref{Prop_SSS}.}
The first part of the proposition follows readily from \Lem s~\ref{Lemma_SSSGGG} and~\ref{Lemma_ESSSGGG}.
Moreover, the second assertion follows from \Lem~\ref{Lemma_SSSGGGlocal}.
Furthermore, we shall establish below that
	\begin{equation}\label{eqPropSSS1}
	\pr\brk{\cS_G=s}\sim\frac1{\sqrt{2\pi}}\exp\bc{-\frac{(s-\Erw(\cS))^2}{2\sigma_\cS^2}}
		\quad\mbox{for any integer $s$ such that $|s-\Erw(\cS_G)|=O(\sqrt{n})$}.
	\end{equation}
This claim implies the third part of the proposition.
For $(s-\Erw(\cS))^2\sigma_\cS^{-2}\sim(\mu_\cS+\lambda_\cS(n_1-\mu_1))^2\sigma_\cS^{-2}$
by \Lem~\ref{Lemma_ESSSGGG} and \Cor~\ref{Cor_VarSSSGGG}, and
$\pr\brk{\cS=s|\order_1=n_1}\sim\pr\brk{\cS_G=s}$ by \Lem~\ref{Lemma_SSSGGG}.
	
To prove~(\ref{eqPropSSS1}) let $\alpha>0$ be arbitrarily small but fixed.
Since $\sigma_\cS^2=\Theta(n)$ by \Lem~\ref{Lemma_VarSSSGGG} and \Cor~\ref{Cor_VarSSSGGG}, 
\Lem~\ref{Lemma_SSSGGGlocal} entails that for a sufficiently small $\beta>0$ and all $s,t$ satisfying
$|s-\Erw(\cS_G)|,|t-\Erw(\cS_G)|\leq n^{0.6}$ and $|s-t|\leq\beta\sigma_\cS$ we have
	\begin{equation}\label{eqPropSSS2}
	(1-\alpha)\pr\brk{\cS_G=s}-n^{-10}
		\leq\pr\brk{\cS_G=t}\leq(1+\alpha)\pr\brk{\cS_G=s}+n^{-10}.
	\end{equation}
Now, suppose that $s$ is an integer such that $|s-\Erw(\cS_G)|\leq O(\sqrt{n})$, and set $z=(s-\Erw(\cS_G))/\sigma_\cS$.
Then \Lem~\ref{Lemma_SSSGGGNormal} implies that
	\begin{equation}\label{eqPropSSS3}
	\pr\brk{|\cS_G-s|\leq\beta\sigma_\cS}\geq\frac{1-\alpha}{\sqrt{2\pi}}\int_{z-\beta}^{z+\beta}\exp(-x^2/2)dx
		\geq(1-2\alpha)\frac{\beta}{\sqrt{2\pi}}\exp(-z^2/2),
	\end{equation}
provided that $\beta$ is small enough.
Furthermore, (\ref{eqPropSSS2}) yields that
	\begin{eqnarray}
	\pr\brk{|\cS_G-s|\leq\beta\sigma_\cS}&=&\sum_{t:|t-s|\leq\beta\sigma_\cS}\pr\brk{\cS_G=t}
		\leq\beta\sigma_\cS((1+\alpha)\pr\brk{\cS_G=s}+n^{-10})\nonumber\\
		&\leq&(1+\alpha)\beta\sigma_\cS\pr\brk{\cS_G=s}+n^{-9},
		\label{eqPropSSS4}
	\end{eqnarray}
because $\sigma_\cS=O(\sqrt{n})$ by \Lem~\ref{Lemma_VarSSSGGG}.
Combining~(\ref{eqPropSSS3}) and~(\ref{eqPropSSS4}), we conclude that
	$$\pr\brk{\cS_G=s}\geq\frac{1-2\alpha}{1+\alpha}\cdot\frac{1}{\sqrt{2\pi}\sigma_\cS}\exp(-z^2/2)-n^{-9}
		\geq\frac{1-4\alpha}{\sqrt{2\pi}\sigma_\cS}\exp\bc{-\frac{(s-\Erw(\cS_G))^2}{2\sigma_\cS^2}}.$$
Since analogous arguments yield the matching upper bound
	$\pr\brk{\cS_G=s}\leq\frac{1+4\alpha}{\sqrt{2\pi}\sigma_\cS}\exp\bc{-\frac{(s-\Erw(\cS_G))^2}{2\sigma_\cS^2}}$,
and because $\alpha>0$ may be chosen arbitrarily small,
we obtain~(\ref{eqPropSSS1}).
\qed

\medskip
Next we will prove Lemma~\ref{Lemma_SSSGGGlocal} which provides the central locality
argument while the more technical proofs of Lemma~\ref{Lemma_SSSGGG}, \ref{Lemma_ESSSGGG} and 
\ref{Lemma_VarSSSGGG} are deferred to the end of this section.

\subsection{Proof of \Lem~\ref{Lemma_SSSGGGlocal}}\label{Sec_SSSGGGlocal}

Since the assertion is symmetric in $s$ and $t$, it suffices to prove that
$\pr\brk{\cS_G=s}\leq(1-\alpha)^{-1}\pr\brk{\cS_G=s}+n^{-10}.$
Let $\cF=E(H_{3,G})\setminus E(H_{2,G})$ be the (random) set of edges added during {\bf R3'}.
We split $\cF$ into three subsets:  let $\cF_1$ consist of all $e\in\cF$ such that either $|e\setminus G|\geq2$ or
	$e$ contains a vertex that belongs to a component of $V\setminus G$ of order $\geq2$.
Moreover, $\cF_2$ is the set of all edges $e\in\cF\setminus\cF_1$ that contain a vertex of $V\setminus G$
	that is also contained in some other edge $e'\in\cF_1$.
Finally, $\cF_3=\cF\setminus(\cF_1\cup\cF_2)$;
thus, all edges $e\in\cF_3$ connect $d-1$ vertices in $G$ with a vertex $v\in V\setminus G$ that
is isolated in $H_{2,G}+\cF_1+\cF_2$, see Figure~\ref{fig:FFF} for an example.
Hence, $H_{3,G}=H_{2,G}+\cF_1+\cF_2+\cF_3$.

\begin{figure}
\begin{center}
\begin{picture}(0,0)%
\includegraphics{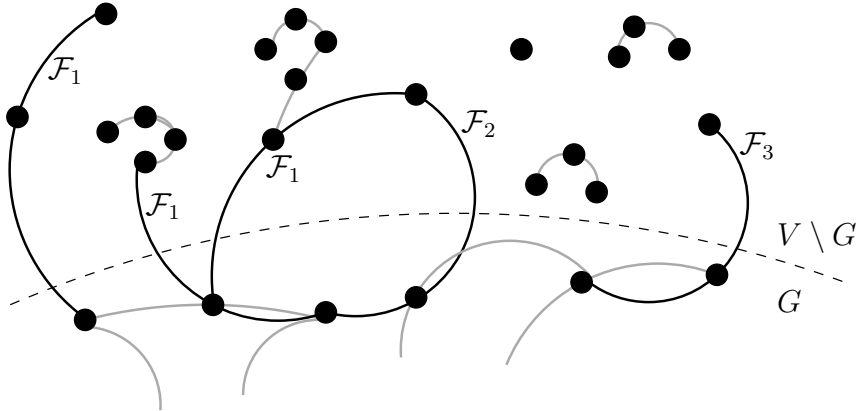}%
\end{picture}%
\setlength{\unitlength}{4144sp}%
\begingroup\makeatletter\ifx\SetFigFont\undefined%
\gdef\SetFigFont#1#2#3#4#5{%
  \reset@font\fontsize{#1}{#2pt}%
  \fontfamily{#3}\fontseries{#4}\fontshape{#5}%
  \selectfont}%
\fi\endgroup%
\begin{picture}(5140,2458)(19,-2416)
\put(4411,-871){\makebox(0,0)[lb]{\smash{{\SetFigFont{12}{14.4}{\familydefault}{\mddefault}{\updefault}{\color[rgb]{0,0,0}$\cF_3$}%
}}}}
\put(271,-421){\makebox(0,0)[lb]{\smash{{\SetFigFont{12}{14.4}{\familydefault}{\mddefault}{\updefault}{\color[rgb]{0,0,0}$\cF_1$}%
}}}}
\put(4636,-1816){\makebox(0,0)[lb]{\smash{{\SetFigFont{12}{14.4}{\familydefault}{\mddefault}{\updefault}{\color[rgb]{0,0,0}$G$}%
}}}}
\put(4636,-1411){\makebox(0,0)[lb]{\smash{{\SetFigFont{12}{14.4}{\familydefault}{\mddefault}{\updefault}{\color[rgb]{0,0,0}$V\setminus G$}%
}}}}
\put(2746,-736){\makebox(0,0)[lb]{\smash{{\SetFigFont{12}{14.4}{\familydefault}{\mddefault}{\updefault}{\color[rgb]{0,0,0}$\cF_2$}%
}}}}
\put(856,-1231){\makebox(0,0)[lb]{\smash{{\SetFigFont{12}{14.4}{\familydefault}{\mddefault}{\updefault}{\color[rgb]{0,0,0}$\cF_1$}%
}}}}
\put(1576,-1006){\makebox(0,0)[lb]{\smash{{\SetFigFont{12}{14.4}{\familydefault}{\mddefault}{\updefault}{\color[rgb]{0,0,0}$\cF_1$}%
}}}}
\end{picture}%
\caption[Edges attaching small components]{The three kinds of edges (black) which attach small components to $G$. The edges
of $H_{2,G}$ are depicted in grey.
The (3-uniform) edges are depicted as circular arcs spanned by the three vertices contained in the corresponding edge.}\label{fig:FFF}
\end{center}
\end{figure}

As a next step, we decompose $\cS_G$ into two contributions corresponding to $\cF_1\cup\cF_2$ and $\cF_3$.
More precisely,  we let $\cS_G^\mybig$ be the number of vertices in $V\setminus G$
that are reachable from $G$ in $H_{2,G}+\cF_1+\cF_2$ and set $\cS_G^\myiso=\cS_G-\cS_G^\mybig$.
Hence, if we let $\cW$ signify the set of all isolated vertices of $H_{2,G}+\cF_1+\cF_2$ in the set $V\setminus G$,
then $\cS_G^\myiso$ equals the number of vertices in $\cW$ that get attached to $G$ via the edges in $\cF_3$.

We can determine the distribution of $\cS_G^\myiso$ precisely.
For if $v\in\cW$, then each edge $e$ containing $v$ and exactly $d-1$ vertices of $G$ is present with probability $p_2$ independently.
Therefore, the probability that $v$ gets attached to $G$ is $1-(1-p_2)^{\bink{n_1}{d-1}}$.
In fact, these events occur independently for all $v\in\cW$.
Consequently,
	\begin{equation}\label{eqmyiso}
	\cS_G^\myiso=\Bin\bc{|\cW|,1-(1-p_2)^{\bink{n_1}{d-1}}},\
		\mu_\myiso=\Erw(\cS_G^\myiso)=|\cW|(1-(1-p_2)^{\bink{n_1}{d-1}})=\Omega(|\cW|),
	\end{equation}
where the last equality sign follows from the fact that $p_2\sim\eps p_1=\Theta(n^{1-d})$.

Hence, $\cS_G=\cS_G^\mybig+\cS_G^\myiso$ features a contribution that satisfies a local limit theorem,
namely the binomially distributed $\cS_G^\myiso$.
Thus, to establish the locality of $\cS_G$ (i.e., \Lem~\ref{Lemma_SSSGGGlocal}), we are going to prove that
$\cS_G$ ``inherits'' the locality of $\cS_G^\myiso$.
To this end, we need to bound $|\cW|$, thereby estimating $\mu_\myiso=\Erw(\cS_G^\myiso)$.
\begin{lemma}\label{Lemma_WWW}
We have $\pr\brk{|\cW|\geq\frac12(n-n_1)\exp(-c)}\geq 1-n^{-10}$.
\end{lemma}
The proof of \Lem~\ref{Lemma_WWW} is just a standard application of Azuma's inequality, cf.~\Sec~\ref{Sec_WWW}.

Further, let $M$ be the set of all triples $(H,F_1,F_2)$ such that
\begin{description}
\item[M1.] $\pr\brk{\cS_G=s|H_{2,G}=H,\,\cF_1=F_1,\,\cF_2=F_2}\geq n^{-11}$, and
\item[M2.] given that $H_{2,G}=H$, $\cF_1=F_1$, and $\cF_2=F_2$, the set $\cW$ has size $\geq\frac12(n-n_1)\exp(-c)=\Omega(n)$.
\end{description}

\begin{lemma}\label{Lemma_localAux}
If $|s-t|\leq\beta\sqrt{n}$ for some small enough $\beta=\beta(\alpha)>0$, then
$\pr\brk{\cS_G=t|(H_{2,G},\cF_1,\cF_2)\in M}\geq(1-\alpha)\pr\brk{\cS_G=s|(H_{2,G},\cF_1,\cF_2)\in M}$.
\end{lemma}
\begin{proof}
Let $(H,F_1,F_2)\in M$, and let $b$ be the value of $\cS_G^\mybig$ given that $H_{2,G}=H$, $\cF_1=F_1$ and $\cF_2=F_2$.
Then given that this event occurs, we have $\cS_G=s$ iff $\cS_G^\myiso=s-b$.
As $(H,F_1,F_2)\in M$, we conclude that
	\begin{eqnarray*}
	\pr\brk{\cS_G=s|H_{2,G}=H,\,\cF_1=F_1,\,\cF_2=F_2}
		&=&\pr\brk{\Bin\bc{|\cW|,1-(1-p_2)^{\bink{n_1}{d-1}}}=s-b}\stacksign{\bf M1}{\geq} n^{-11}.
	\end{eqnarray*}
Therefore, the Chernoff bound~\eqref{eqChernoff} implies that $|s-b-\mu_\myiso|\leq n^{0.6}$.
Furthermore, since we assume that $|t-s|\leq\beta n^{1/2}$ for some small $\beta=\beta(\alpha)>0$
and as $\mu_\myiso=|\cW|(1-(1-p_2)^{\bink{n_1}{d-1}})\geq\Omega(n)$ due to {\bf M2},
\Prop~\ref{Prop_Bin} entails that
	$$\pr\brk{\Bin\bc{|\cW|,1-(1-p_2)^{\bink{n_1}{d-1}}}=t-b}\geq
			(1-\alpha)\pr\brk{\Bin\bc{|\cW|,1-(1-p_2)^{\bink{n_1}{d-1}}}=s-b}.$$
Thus, the assertion follows from~(\ref{eqmyiso}).
\qed\end{proof}

\noindent
\emph{Proof of \Lem~\ref{Lemma_SSSGGGlocal}.}
By \Lem s~\ref{Lemma_WWW} and~\ref{Lemma_localAux}, we have
	\begin{eqnarray*}
	\pr\brk{\cS_G=s}&\leq&\pr\brk{\cS_G=s|(H_{2,G},\cF_1,\cF_2)\not\in M}\pr\brk{(H_{2,G},\cF_1,\cF_2)\not\in M}
		+(1-\alpha)^{-1}\pr\brk{\cS_G=t}\\
		&\stacksign{\bf M1,\,M2}{\leq}&n^{-11}+\pr\brk{|\cW|=o(n)}+(1-\alpha)^{-1}\pr\brk{\cS_G=t}
			\leq(1-\alpha)^{-1}\pr\brk{\cS_G=t}+n^{-10},
	\end{eqnarray*}
as claimed.
\qed

\subsection{Proof of \Lem~\ref{Lemma_SSSGGG}}\label{Sec_SSSGGG}

Let $\cL_G$ signify the event that $G$ is the largest component of $H_1$.
Given that $\cL_G$ occurs, the edges in $H_3-G$ do not occur independently anymore.
For if $\cL_G$ occurs, then $H_1-G$ does not contain a component on more than $|G|$ vertices.
Nonetheless, the following lemma shows that if $E\subset\cE(V)\setminus\cE(G)$ is a set of edges such that the hypergraph
$H(E)=(V,E\cap\cE(V\setminus G))$ does not feature a ``big'' component, then the dependence of the edges is very small.
In other words, the probability that the edges $E$ are present in $H_3$ is very close to the probability that these
edges are present in the ``artificial'' model $H_{3,G}$, in which edges occur independently.
\begin{lemma}\label{Lemma_SSSGGGAux1}
For any set $E\subset\cE(V)\setminus\cE(G)$
such that $\order(H(E))\leq\ln^2n$ we have
	$$\pr\brk{E(H_3)\setminus\cE(G)=E\,|\,\cL_G}=(1+O(n^{-10}))\pr\brk{E(H_{3,G})=E}.$$
\end{lemma}
Before getting down to the proof of \Lem~\ref{Lemma_SSSGGGAux1}, we first show how it implies \Lem~\ref{Lemma_SSSGGG}.
As a first step, we derive that it is actually quite unlikely that either $H_3-G$ or $H_{3,G}-G$ features a component
on $\geq\ln^2n$ vertices.

\begin{corollary}\label{Lemma_SSSGGGAux2}
We have
	$\pr\brk{\order(H_3-G)>\ln^2n|\cL_G},\pr\brk{\order(H_{3,G}-G)>\ln^2n}=O(n^{-10}).$
\end{corollary}
\begin{proof}
\Thm~\ref{Thm_global} implies that $\pr\brk{\order(H_{3,G}-G)>\ln^2n}=O(n^{-10})$,
because $H_{3,G}$ simply is a random hypergraph $H_d(n-n_1,p)$, and $\bink{n-n_1}{d-1}p\sim\bink{n-\mu_1}{d-1}p<(d-1)^{-1}$
by~\eqref{eqrhoupper}.
Hence, \Lem~\ref{Lemma_SSSGGGAux1} yields that
	$\pr\brk{\order(H_3-G)\leq\ln^2n|\cL_G}\geq(1-O(n^{-10}))\pr\brk{\order(H_{3,G}-G)\leq\ln^2n}\geq1-O(n^{-10})$.
\qed\end{proof}

\noindent\emph{Proof of \Lem~\ref{Lemma_SSSGGG}.}
Let $\cA_s$ denote the set of all subsets $E\subset\cE(V)\setminus\cE(G)$ such that in the hypergraph $(V,E)$ exactly
$s$ vertices in $V\setminus G$ are reachable from $G$.
Moreover, let $\cB_s$ signify the set of all $E\in\cA_s$ such that $\order(H(E))\leq\ln^2n$.
Then
	\begin{equation}\label{eqSSSGGGI}
	\pr\brk{\cS=s|\cL_G}=\pr\brk{E(H_3)\setminus\cE(G)\in\cA_s|\cL_G}\mbox{, and }
		\pr\brk{\cS_G=s}=\pr\brk{E(H_{3,G})\in\cA_s}.
	\end{equation}
Furthermore, by \Cor~\ref{Lemma_SSSGGGAux2}
	\begin{eqnarray}
	\pr\brk{E(H_3)\setminus\cE(G)\in\cA_s\setminus\cB_s|\cL_G}
		&\leq&\pr\brk{\order(H_3-G)>\ln^2n|\cL_G}=O(n^{-10}),\label{eqSSSGGGII}\\
	\pr\brk{E(H_{3,G})\in\cA_s\setminus\cB_s}
		&\leq&\pr\brk{\order(H_{3,G}-G)>\ln^2n}=O(n^{-10}).\label{eqSSSGGGIII}
	\end{eqnarray}
Combining~(\ref{eqSSSGGGI}), (\ref{eqSSSGGGII}), and~(\ref{eqSSSGGGIII}), we conclude that
	\begin{eqnarray*}
	\pr\brk{\cS=s|\cL_G}&=&\pr\brk{E(H_3)\setminus\cE(G)\in\cB_s|\cL_G}+O(n^{-10})\\
		&\stacksign{\Lem~\ref{Lemma_SSSGGGAux1}}{=}&
		\pr\brk{E(H_{3,G})\in\cB_s}+O(n^{-10})
		=\pr\brk{\cS_G=s}+O(n^{-10}),
	\end{eqnarray*}
thereby completing the proof.
\qed

\medskip
Thus, the remaining task is to prove \Lem~\ref{Lemma_SSSGGGAux1}.
To this end, let $\cH_1(E)$ denote the event that $\cE(V\setminus G)\cap E(H_1)=E$.
Moreover, let $\cH_2(E)$ signify the event that $\cE(V\setminus G)\cap E(H_2)\setminus E(H_1)=E$
(i.e., $E$ is the set of edges added during {\bf R2}).
Further, let $\cH_3(E)$ be the event that $\cE(G,V\setminus G)\cap E(H_3)=E$
(i.e.,  $E$ consists of all edges added by {\bf R3}).
In addition, define events $\cH_{1,G}(E)$, $\cH_{2,G}(E)$, $\cH_{3,G}(E)$ analogously,
with $H_1$, $H_2$, $H_3$ replaced by $H_{1,G}$, $H_{2,G}$, $H_{3,G}$.
Finally, let $\comp_G$ denote the event that $G$ is a component of $H_1$.
In order to prove \Lem~\ref{Lemma_SSSGGGAux1}, we establish the following.

\begin{lemma}\label{Lemma_SSSGGGAux3}
Let $E_1\subset\cE(V\setminus G)$, $E_2\subset\cE(V\setminus G)\setminus E_1$, and
$E_3\subset\cE(G,V\setminus G)$.
Moreover, suppose that $\order(H(E_1))\leq\ln^2n$.
Then
	$\pr\brk{\bigwedge_{i=1}^3\cH_i(E_i)|\cL_G}=(1+O(n^{-10}))\pr\brk{\bigwedge_{i=1}^3\cH_{i,G}(E_i)}.$
\end{lemma}
\begin{proof}
Clearly,
	\begin{eqnarray}
	\pr\brk{\bigwedge_{i=1}^3\cH_i(E_i)|\cL_G}
		&=&\frac{\pr\brk{\cH_2(E_2)\wedge\cH_3(E_3)|\cL_G\wedge\cH_1(E_1)}
				\pr\brk{\cH_1(E_1)\wedge\cL_G}}{\pr\brk{\cL_G}}.
	\label{eqSSSGGGAux31}
	\end{eqnarray}
Furthermore, since {\bf R2} and {\bf R3} add edges independently of the 1st round with probability $p_2$,
and because the same happens during {\bf R2'} and {\bf R3'}, we have
	\begin{eqnarray}
		\pr\brk{\cH_2(E_2)\wedge\cH_3(E_3)|\cL_G\wedge\cH_1(E_1)}
			&=&\pr\brk{\cH_{2,G}(E_2)\wedge\cH_{3,G}(E_3)|\cH_{1,G}(E_1)}.
		\label{eqSSSGGGAux32}
	\end{eqnarray}

Moreover, given that $\cH_1(E_1)$ occurs, $H_1-G$ has no component on more than $\ln^2n$ vertices.
Hence, $G$ is the largest component of $H_1$ iff $G$ is a component; that is, given that $\cH_1(E_1)$ occurs,
the events $\cL_G$ and $\comp_G$ are equivalent.
Therefore, $\pr\brk{\cL_G\wedge\cH_1(E_1)}=\pr\brk{\comp_G\wedge\cH_1(E_1)}$.
Further, whether or not $G$ is a component of $H_1$ is independent of the edges contained in $V\setminus G$, and thus
$\pr\brk{\comp_G\wedge\cH_1(E_1)}=\pr\brk{\comp_G}\pr\brk{\cH_1(E_1)}$.
Hence, as each edge in $E_1$ is present in $H_1$ as well as in $H_{1,G}$ with probability $p_1$ independently, we obtain
	\begin{eqnarray}
	\pr\brk{\cL_G\wedge\cH_1(E_1)}
		&=&\pr\brk{\comp_G}p_1^{|E_1|}(1-p_1)^{\cE(V\setminus G)-|E_1|}
		=\pr\brk{\comp_G}\pr\brk{\cH_{1,G}(E_1)}.
		\label{eqSSSGGGAux33}
	\end{eqnarray}

Combining~(\ref{eqSSSGGGAux31}), (\ref{eqSSSGGGAux32}), and~(\ref{eqSSSGGGAux33}), we obtain
	\begin{eqnarray}
	\pr\brk{\bigwedge_{i=1}^3\cH_i(E_i)|\cL_G}
		&=&\frac{\pr\brk{\comp_G}}{\pr\brk{\cL_G}}\cdot\pr\brk{\bigwedge_{i=1}^3\cH_{i,G}(E_i)}.
	\label{eqSSSGGGAux34}
	\end{eqnarray}
Since by \Thm~\ref{Thm_global} with probability $\geq1-n^{-10}$ the random hypergraph $H_1=H_d(n,p_1)$
has precisely one component of order $\Omega(n)$, we get $\frac{\pr\brk{\comp_G}}{\pr\brk{\cL_G}}=1+O(n^{-10})$.
Hence, (\ref{eqSSSGGGAux34}) implies the assertion.
\qed\end{proof}

\noindent\emph{Proof of \Lem~\ref{Lemma_SSSGGGAux1}.}
For any set $E\subset\cE(V)\setminus\cE(G)$ let $\cF(E)$ denote the set of all decompositions $(E_1,E_2,E_3)$
of $E$ into three disjoint sets such that $E_1,E_2\subset\cE(V\setminus G)$ and $E_3\subset\cE(G,V\setminus G)$.
If $\order(H(e))\leq\ln^2n$, then \Lem~\ref{Lemma_SSSGGGAux3}  implies that
	\begin{eqnarray*}
	\pr\brk{E(H_3)\setminus\cE(G)=E|\cL_G}
		&=&\sum_{(E_1,E_2,E_3)\in\cF(E)}
			\pr\brk{\bigwedge_{i=1}^3\cH_i(E_i)|\cL_G}\\
		&\hspace{-7.1cm}=&\hspace{-3.5cm}(1+O(n^{-10}))\sum_{(E_1,E_2,E_3)\in\cF(E)}\pr\brk{\bigwedge_{i=1}^3\cH_{i,G}(E_i)}
		=(1+O(n^{-10}))\pr\brk{E(H_{3,G})=E},
	\end{eqnarray*}
as claimed.
\qed

\subsection{Proof of \Lem~\ref{Lemma_ESSSGGG}}\label{Sec_ESSSGGG}

Recall that $S_G$ signifies the set of all vertices $v\in V\setminus G$ that are reachable from $G$ in $H_{3,G}$, so that $\cS_G=|S_G|$.
Letting $\comp_v$ denote the component of $H_{2,G}$ that contains $v\in V$, we have
	\begin{eqnarray}\label{eqESSSGGG1}
	\Erw(\cS_G)&=&\sum_{v\in V\setminus G}\pr\brk{v\in S_G}
		=\sum_{v\in V\setminus G}\sum_{k=1}^{n-n_1}\pr\brk{v\in S_G||\comp_v|=k}\pr\brk{|\comp_v|=k}
	\end{eqnarray}
Since $H_{2,G}$ is just a random hypergraph $H_d(n-n_1,p)$, and because $\bink{n-n_1}{d-1}p\sim\bink{n-\mu_1}{d-1}p<(d-1)^{-1}$
by~\eqref{eqrhoupper},  \Thm~\ref{Thm_global} entails that $\order(H_{2,G})\leq\ln^2n$ with probability $\geq1-n^{-10}$.
Therefore, (\ref{eqESSSGGG1}) yields
	\begin{eqnarray}\label{eqESSSGGG2}
	\Erw(\cS_G)&=&o(1)+\sum_{v\in V\setminus G}
			\sum_{1\leq k\leq\ln^2n}\pr\brk{v\in S_G||\comp_v|=k}\pr\brk{|\comp_v|=k}.
	\end{eqnarray}
To estimate $\pr\brk{v\in S_G||\comp_v|=k}$,
let $z=z(n_1)=(n_1-\mu_1)/\sigma_1$, 
$\xi_0=\exp\brk{-p_2\brk{\bink{n-1}{d-1}-\bink{n-\mu_1}{d-1}}}$, and
	$\xi(z)=\xi_0\brk{1+z\sigma_1p_2\bink{n-\mu_1}{d-2}}$.
Additionally, let $\zeta(z)=\bink{n-n_1}{d-1}p\sim\bink{n-\mu_1}{d-1}p-z\sigma_1\bink{n-\mu_1}{d-2}p$.

\begin{lemma}\label{Lemma_AttachProb}
For all $1\leq k\leq\ln^2n$ we have $\pr\brk{v\in S_G\;|\;|\comp_v|=k}=1-\xi(z)^k+\tO{n^{-1}}$.
\end{lemma}
\begin{proof}
Suppose that $|\comp_v|=k$ but $v\not\in S_{G}$.
This is the case iff in $H_{3,G}$ there occurs no edge that is incident to both $G$ and $\comp_v$.
Letting $\cE(G,\comp(v))$ denote the set of all possible edges connecting $G$ and $\comp_v$,
we shall prove below that
	\begin{eqnarray}\nonumber
	|\cE(G,\comp_v)|=k\brk{\bink{n}{d-1}-\bink{n-\mu_1}{d-1}+\frac{z\sigma_1}{d-1}\bink{n-\mu_1}{d-2}}+
		\tO{n^{d-2}}\\
		&=&\tO{n^{d-1}}.\label{eqAttachProb1}
	\end{eqnarray}
By construction every edge in $\cE(G,\comp_v)$ occurs in $H_{3,G}$ with probability $p_2$ independently.
Therefore,
	\begin{eqnarray*}
	\pr\brk{v\not\in S_G||\comp_v|=k}&=&(1-p_2)^{|\cE(G,\comp_v)|}
		=(1+\tO{n^{-1}})\exp\brk{-p_2|\cE(G,\comp_v)|}\\
		&\stacksign{(\ref{eqAttachProb1})}{=}&
			(1+\tO{n^{-1}})\xi(z)^k,
	\end{eqnarray*}
hence the assertion follows.

Thus, the remaining task is to prove~(\ref{eqAttachProb1}).
As a first step, we show that
	\begin{equation}\label{eqAttachProb2}
	|\cE(G,\comp_v)|=\bink{n}d-\bink{n-k}d-\bink{n-n_1}{d}+\bink{n-n_1-k}d.
	\end{equation}
For there are $\bink{n}d$ possible edges in total, among which $\bink{n-k}d$ contain no vertex of $\comp_v$,
$\bink{n-n_1}d$ contain no vertex of $G$, and $\bink{n-n_1-k}d$ contain neither a vertex of $\comp_v$ nor of $G$;
thus, (\ref{eqAttachProb2}) follows from the inclusion/exclusion formula.
Furthermore, as $k=\tOl$, we have
	$\bink{n}d-\bink{n-k}d=(1+\tO{n^{-1}})k\bink{n}{d-1}$ and
	$\bink{n-n_1}d-\bink{n-n_1-k}d=
		(1+\tO{n^{-1}})k\bink{n-n_1}{d-1}$.
Thus (\ref{eqAttachProb2}) yields
	\begin{equation}\label{eqAttachProb3}
	|\cE(G,\comp(v))|=(1+\tO{n^{-1}})k\brk{\bink{n}{d-1}-\bink{n-n_1}{d-1}}.
	\end{equation}
As $n_1=\mu_1+z\sigma_1$, we have $\bink{n-n_1}{d-1}=\bink{n-\mu_1}{d-1}-z\sigma_1\bink{n-n_1}{d-2}+\tO{n^{d-2}}$,
so that (\ref{eqAttachProb1}) follows from~(\ref{eqAttachProb3}).
\qed\end{proof}

Let $q(\zeta,\xi)=\sum_{k=1}^{\infty}q_k(\zeta)\xi^k$ be the function from \Prop~\ref{Prop_branching}.
Combining~(\ref{eqESSSGGG2}) with \Prop~\ref{Prop_branching} and \Lem~\ref{Lemma_AttachProb}, we conclude that
	\begin{eqnarray}\label{eqESSSGGG3}
	\Erw(\cS_G)&=&o(n^{1/2})+q((n-n_1)p,\xi(z))(n-n_1)=o(n^{1/2})+q(\zeta(z),\xi(z))(n-n_1).
	\end{eqnarray}
Since $q$ is differentiable (cf.~\Prop~\ref{Prop_branching}), we let
	$\Delta_\zeta=\frac{\partial q}{\partial\zeta}(\zeta(0),\xi(0))$ and
	$\Delta_\xi=\frac{\partial q}{\partial\xi}(\zeta(0),\xi(0))$.
As $\zeta(z)-\zeta(0),\xi(z)-\xi(0)=O(n^{-1/2})$, we get
	\begin{eqnarray}\nonumber
		q(\zeta(z),\xi(z))-q(\zeta(0),\xi(0))&=&
		(\zeta(z)-\zeta(0))\Delta_\zeta+
		(\xi(z)-\xi(0))\Delta_\xi+o(n^{-1/2})\\
		&=&z\sigma_1\bink{n-\mu_1}{d-2}\brk{\xi_0\Delta_\xi p_2-\Delta_\zeta p}+o(n^{-1/2}).
	\label{eqESSSGGG4}
	\end{eqnarray}
Finally, let $\mu_\cS=(n-\mu_1)q(\zeta(0),\xi(0))$ and
	$\lambda_\cS=q(\zeta(0),\xi(0))-(d-1)\brk{\eps\xi_0\Delta_\xi-\Delta_\zeta}\bink{n-\mu_1}{d-1}p.$
Then combining~(\ref{eqESSSGGG3}) and~(\ref{eqESSSGGG4}), we see that
	$\Erw(\cS_G)=\mu_\cS+z\sigma_1\lambda_\cS+o(\sqrt{n})$,
as desired.

\subsection{Proof of \Lem~\ref{Lemma_VarSSSGGG}}\label{Sec_VarSSSGGG}

Remember that $S_G$ denotes the set of all ``attached'' vertices, and $N_{v,G}$ 
the order of the component of $v\in V\setminus G$ in the graph $H_{2,G}$.

The following lemma provides an asymptotic formula for $\Var(\cS_G)$.

\begin{lemma}\label{Lemma_VarFormula}
Let $r_{G,i}=\pr\brk{N_{v,G}=i\wedge v\in S_G}$ and $\bar r_{G,i}=\pr\brk{N_{v,G}=i\wedge v\not\in S_G}$
for any vertex $v\in V\setminus G$.
Moreover, set
	$r_G=\sum_{i=1}^Lr_{G,i},\ R_G=\sum_{i=1}^Lir_{G,i},\ \bar R_G=\sum_{i=1}^Li\bar r_{G,i}$ for $L=\ug{\ln^2 n}$.
In addition, let $\alpha_G=1-|G|/n$ and
\begin{equation}\label{eq:Gamma}
\Gamma_G=(1-R_G)(R_G-r_G)+((d-1)c-1)\frac{R_G^2}{r_G}+R_G+(d-1)(1-\alpha_G^{d-2})\eps c\bar R_G^2+\frac{1-\alpha_G^{d-2}}{1-\alpha_G^{d-1}}\bar R_G.
\end{equation}
Then $\Var(\cS_G)\sim\alpha_G^2\Gamma_G n+\alpha_G r_G(1-r_G)n$.
\end{lemma}
Before we get down to the proof of \Lem~\ref{Lemma_VarFormula}, we observe that it implies \Lem~\ref{Lemma_VarSSSGGG}.

\medskip
\noindent
\emph{Proof of \Lem~\ref{Lemma_VarSSSGGG}.}
By Theorem~\ref{Thm_global} part 2 together with \Lem~\ref{Lemma_SSSGGG} we know that with probability at least $1-n^{-8}$ there
are no components of order $>\ln^2 n$ inside of $V\setminus G$.
Let $q(\zeta,\xi)=\sum_{k=1}^{\infty}q_k(\zeta)\xi^k$ be the function from \Prop~\ref{Prop_branching}, and let $\xi(z)$ be as in \Lem~\ref{Lemma_AttachProb}.
Then \Prop~\ref{Prop_branching} and \Lem~\ref{Lemma_AttachProb} entail that for all $v\in V\setminus G$
	\begin{eqnarray*}
	r_{G,i}=q_i\bc{\bink{n-|G|}{d-1}p}\xi((|G|-\mu_1)/\sigma_1),&&
	\bar r_{G,i}\sim q_i\bc{\bink{n-|G|}{d-1}p}(1-\xi((|G|-\mu_1)/\sigma_1)).
	\end{eqnarray*}

By (\ref{eqqkcbound}) there exists a number $0<\gamma<1$ such that $q_i\bc{\bink{n-|G|}{d-1}p}\leq\gamma^i$.
Since $0\leq \xi((|G|-\mu_1)/\sigma_1)\leq 1$, this yields $r_{G,i},\bar r_{G,i}\leq\gamma^i$.
Hence,  $R_G,\bar R_G=O(1)$, so that
\Lem~\ref{Lemma_VarFormula} implies $\Var(\cS_G)=O(n)$.

Finally, if $G'\subset V$ satisfies $||G'|-|G||\leq n^{0.9}$, then $|\bink{n-|G|}{d-1}p-\bink{n-|G'|}{d-1}p|=O(|G|-|G'|)/n$, because $p=O(n^{1-d})$.
Therefore, $|q_i\bc{\bink{n-|G|}{d-1}p}-q_i\bc{\bink{n-|G'|}{d-1}p}|=O(|G|-|G'|)/n$, because the function $\zeta\mapsto q_i(\zeta)$ is differentiable.
Similarly, as $\xi(z)=\xi_0(1+z\sigma_1p_2\bink{n-\mu_1}{d-2})$ for some fixed $\xi_0=\Theta(1)$,
we have $|\xi((|G|-\mu_1)/\sigma_1)-\xi((|G'|-\mu_1)/\sigma_1)|=O(|G|-|G'|)/n$.
Consequently, $|r_{G,i}-r_{G',i}|=O(|G|-|G'|)/n$ and $|\bar r_{G,i}-\bar r_{G',i}|=O(|G|-|G'|)/n$, and thus
	$$|r_G-r_{G'}|,|R_G-R_{G'}|, |\bar R_G-\bar R_{G'}|=O(|G|-|G'|)/n=O(n^{-0.1}).$$
Hence, \Lem~\ref{Lemma_VarFormula} implies that $|\Var(\cS_G)-\Var(\cS_{G'})|=o(n)$.
\qed

The remaining task is to establish \Lem~\ref{Lemma_VarFormula}.
Keeping $G$ fixed, in the sequel we constantly omit the subscript $G$ in order to ease up the notation; thus, we write $\alpha$ instead of $\alpha_G$ etc.
As a first step, we compute $\pr(v,w\in S)-r^2$.
Setting
	\begin{eqnarray*}
	S_1&=&\sum_{i,j=1}^L\brk{\pr\brk{N_w=j\wedge w\in S|w\not\in C_v,N_v=i,v\in S}
			-\pr\brk{N_w=j\wedge w\in S}}\\&&\qquad\qquad\times\pr\brk{w\not\in C_v|N_v=i,v\in S}\pr\brk{N_v=i\wedge v\in S},\\
	S_2&=&(1-r)\sum_{i=1}^L\pr\brk{w\in C_v|N_v=i,v\in S}\pr\brk{N_v=i\wedge v\in S},
	\end{eqnarray*}
we have $\pr(v,w\in S)-r^2=S_1+S_2$.

To compute $S_2$, observe that whether $w\in C_v$ depends only on $N_v$, but not on the event $v\in S$.
Therefore,
	$\pr\brk{w\in C_v|N_v=i,v\in S}=\pr\brk{w\in C_v|N_v=i}=\bink{n-2}{i-2}\bink{n-1}{i-1}^{-1}=\frac{i-1}{n-1},$
because given that $N_v=i$, there are $\bink{n_0-1}{i-1}$ ways to choose the set $C_v\subset V\setminus G$,
while there are $\bink{n_0-2}{i-2}$ ways to choose $C_v$ in such a way that $w\in C_v$.
As a consequence,
	\begin{eqnarray*}
	S_2&\sim&\frac{1-r}{n-1}\sum_{i=1}^L(i-1)\pr\brk{N_v=i\wedge v\in S}=\frac{1-r}{n-1}(R-r).
	\end{eqnarray*}	

With respect to $S_1$, we let
	\begin{eqnarray*}
	P_1(i,j)&=&\pr\brk{N_w=j|w\not\in C_v,N_v=i},\\
	P_2(i,j)&=&\pr\brk{w\in S|N_w=j,w\not\in C_v,N_v=i,v\in S},
	\end{eqnarray*}
so that
	\begin{eqnarray*}
	S_1&=&\sum_{i,j}\brk{P_1(i,j)P_2(i,j)-\pr\brk{N_w=j\wedge w\in S}}\pr\brk{w\not\in C_v|N_v=i,v\in S}\pr\brk{N_v=i\wedge v\in S}\\
		&\sim&\sum_{i,j}\brk{P_1(i,j)P_2(i,j)-\pr\brk{N_w=j}\pr\brk{w\in S|N_w=j}}\pr\brk{N_v=i\wedge v\in S}.
	\end{eqnarray*}

\begin{lemma}\label{Lemma_CompVar1}
We have $P_1(i,j)\pr\brk{N_w=j}^{-1}=1+\frac{((d-1)c-1)ij+i}{n-n_1}+\tO{n^{-2}}$.
\end{lemma}
\begin{proof}\ 
This argument is similar to the one used in the proof of Lemma~41 in \cite{CMS04}.
Remember that if we restrict our view on $H_{3,G}$ to the set $V\setminus G$
the hypergraph is similar to a $H_d(n-n_1,p)$.
In order to estimate $S_1$, we observe that 
	\begin{equation}\label{eqcorrv6} 
	\pr\brk{N_w=j\textrm{ in }H_d(n-n_1,p)\,|\,N_v=i,\,w\not\in C_v}=\pr\brk{N_w=j\textrm{ in }H_d(n-n_1,p)\setminus C_v}. 
	\end{equation} 
Given that $N_v=i$, $\hnp\setminus C_v$ is distributed as a random hypergraph $H_d(n-n_1-i,p)$. 
Hence, the probability that $N_w=j$ in $\hnp\setminus C_v$ equals the probability 
that a given vertex of $H_d(n-n_1-i,p)$ belongs to a component of order $j$. 
Therefore, we can compare $\pr\brk{N_w=j\textrm{ in }H_d(n-n_1,p)\setminus C_v}$ and 
$\pr\brk{N_w=j\textrm{ in }H_d(n-n_1,p)}$ as follows: 
in $H_d(n-n_1-i,p)$ there are $\bink{n-n_1-i-1}{j-1}$ ways to choose the set $C_w\setminus\{j\}$. 
Moreover, there are 
	$\bink{n-n_1-i}d-\bink{n-n_1-i-j}d-\bink{j}d$ 
possible edges connecting the chosen set $C_w$ with $V\setminus C_w$, and as $C_w$ is a component, 
none of these edges is present. 
Since each such edge is present with probability $p$ independently, the probability that there 
is no $C_w$-$V\setminus C_w$ edge equals 
	$$(1-p)^{\bink{n-n_1-i}d-\bink{n-n_1-i-j}d-\bink{j}d}.$$ 
By comparison, in $H_d(n-n_1,p)$ there are $\bink{n-n_1-1}{j-1}$ ways to choose the vertex set of $C_w$. 
Further, there are $\bink{n-n_1}d-\bink{n-n_1-j}d-\bink{j}d$ possible edges connecting 
$C_w$ and $V\setminus C_w$, each of which is present with probability $p$ independently. 
Thus, letting 
	$\gamma=\bink{n-n_1-i}d-\bink{n-n_1-i-j}d-\brk{\bink{n-n_1}d-\bink{n-n_1-j}d},$
we obtain 
	\begin{eqnarray} 
	\frac{\pr\brk{N_w=j\textrm{ in }H_d(n-n_1,p)\setminus C_v}}{\pr\brk{N_w=j\textrm{ in }H_d(n-n_1,p)}} 
		&=&\bink{n-n_1-i-1}{j-1}\bink{n-n_1-1}{j-1}^{-1}(1-p)^{\gamma}. 
		\label{eqcorrv3}
	\end{eqnarray} 
Concerning the quotient of the binomial coefficients, we have
	\begin{equation}\label{eqcorrv4} 
	\bink{n-n_1-i-1}{j-1}\bink{n-n_1-1}{j-1}^{-1}=\exp\brk{-\frac{i(j-1)}{n-n_1}+\tO{n^{-2}}}.
	\end{equation}
Moreover,
	$\gamma=\bink{n-n_1}d\brk{\frac{(n-n_1-i)_d+(n-n_1-j)_d-(n-n_1-i-j)_d}{(n-n_1)_d}-1}$. 
Expanding the falling factorials, we get 
	\begin{eqnarray}	 
	\gamma=\bink{n-n_1}d\brk{\frac{\bink{d}2(i^2+j^2-(i+j)^2)}{(n-n_1)^2}+\tO{n^{-3}}} 
		=-\bink{n-n_1}{d-2}ij+\tO{n^{d-3}}.\nonumber\\\label{eqcorrv5} 
	\end{eqnarray} 
Plugging~(\ref{eqcorrv4}) and~(\ref{eqcorrv5}) into~(\ref{eqcorrv3}), we obtain 
	\begin{eqnarray*} 
	\frac{\pr\brk{N_w=j\textrm{ in }H_d(n-n_1,p)\setminus C_v}}{\pr\brk{N_w=j\textrm{ in }H_d(n-n_1,p)}}&=& 
		\exp\brk{-\frac{i(j-1)}{n-n_1}+\tO{n^{-2}}}(1-p)^{-\bink{n-n_1}{d-2}ij+\tO{n^{d-3}}}\\ 
		&=&
		\exp\brk{-\frac{i(j-1)}{n-n_1}+\bink{n-n_1}{d-2}ijp+\tO{n^{-2}}}\\ 
		&=&1+(n-n_1)^{-1}\brk{((d-1)c-1)ij+i}+\tO{n^{-2}}. 
	\end{eqnarray*} 
Therefore, by~(\ref{eqcorrv6}) 
	\begin{eqnarray}\nonumber 
	\pr\brk{N_w=j|N_v=i,\,w\not\in C_v}-\pr\brk{N_w=j\textrm{ in }H_d(n-n_1,p)}\\ 
		&\hspace{-12cm}=&\hspace{-6cm} 
		\pr\brk{N_w=j\textrm{ in }H_d(n-n_1,p)}\brk{n^{-1}\brk{((d-1)c-1)ij+i}+\tO{n^{-2}}}. 
		\label{eqcorrv7} 
	\end{eqnarray} 
\qed\end{proof}

\begin{lemma}\label{Lemma_CompVar2}
Setting
$\gamma_1=\frac{1-\alpha^{d-2}}{\pr\brk{v\in S|N_v=i}(1-\alpha^{d-1})}$ and $\gamma_2=(d-1)(1-\alpha^{d-2})\eps c$,
we have
	$P_2(i,j)-\pr\brk{w\in S|N_w=j}
		=n^{-1}\pr\brk{w\not\in S|N_w=j}(j\gamma_1-ij\gamma_2)+\tO{n^{-2}}.$
\end{lemma}
\begin{proof}
Let $\cF$ be the event that $N_w=j$, $w\not\in C_v$, $N_v=i$, and $v\in S$.
Moreover, let $\cQ$ be the event that in $H_3$ there exists an edge incident to the three sets $C_v$, $C_w$, and $G$ simultaneously, so that
$P_2(i,j)=\pr\brk{\cQ|\cF}+\pr\brk{w\in S|\neg\cQ,\cF}\pr\brk{\neg\cQ|\cF}.$

To bound $\pr\brk{w\in S|\neg\cQ,\cF}-\pr\brk{w\in S|N_w=j}$, we condition on the event that $C_v$ and $C_w$ are fixed disjoint  sets of sizes $i$ and $j$.
Let $Q'$ signify the probability that $C_w$ is reachable from $G$ in $H_{3,G}$, and let
$Q$ denote the probability that $C_w$ is reachable from $G$ in $H_{3,G}$, and that the event $\neg\cQ$ occurs.
Then $Q'$ corresponds to $\pr\brk{w\in S|N_w=j}$ and $Q$ to $\pr\brk{w\in S|\neg\cQ,\cF}$, so that our aim is to estimate $Q-Q'$.
As there are $|\cE(G,C_v)|-|\cE(G,C_v,C_w)|$ possible edges that join $C_v$ and $G$ but avoid $C_w$, each of which
is present in $H_{3,G}$ with probability $p_2$ independently, we have
	$$Q=1-(1-p_2)^{|\cE(G,C_v)|-|\cE(G,C_v,C_w)|},
		\mbox{ while }Q'=1-(1-p_2)^{|\cE(G,C_w)|}.$$
Therefore,
	\begin{eqnarray*}
	Q-Q'&=&(1-p_2)^{|\cE(G,C_w)|}\brk{1-(1-p_2)^{-|\cE(G,C_v,C_w)|}}\\
		&\sim&(1-Q')\bc{1-\exp\brk{p_2|\cE(G,C_v,C_w)|}}\sim ij(Q'-1)\brk{\bink{n}{d-2}-\bink{n_0}{d-2}}p_2.
	\end{eqnarray*}
As $\bink{n-1}{d-1}p_2\sim\eps c$, we thus get
	\begin{equation}\label{eqCompVar21}
	\pr\brk{w\in\cS|\neg\cQ,\cF}-\pr\brk{w\in S|N_w=j}\sim ij(\pr\brk{w\in S|N_w=j}-1)(d-1)(1-\alpha^{d-2})\eps cn^{-1}.
	\end{equation}

With respect to $\pr\brk{\cQ|\cF}$, we let $\cK$ signify the number of edges joining $C_v$ and $G$.
Given that $\cF$ occurs, $\cK$ is asymptotically Poisson with mean
	$\lambda_i=i\brk{\bink{n}{d-1}-\bink{n_0}{d-1}}p_2\sim i(1-\alpha^{d-1})\eps c.$
Moreover, given that $\cK=k$, the probability that one of these $k$ edges hits $C_w$ is $\cP(k)\sim\frac{k\cE(G,C_v,C_w)}{\cE(C_v,G)}$, and thus
	\begin{eqnarray*}
	\cP(k)&\sim&jk\brk{\bink{n}{d-2}-\bink{n_0}{d-2}}\brk{\bink{n}{d-1}-\bink{n_0}{d-1}}^{-1}\sim
		jk(d-1)\frac{1-\alpha^{d-2}}{1-\alpha^{d-1}}.
	\end{eqnarray*}
Consequently,
	\begin{eqnarray}\label{eqCompVar22}
	\pr\brk{\cQ|\cF}&\sim&\frac{\exp(-\lambda_i)}{1-\exp(-\lambda_i)}\sum_{k\geq1}\frac{jk\lambda_i^k}{k!}\cP(k)
			\sim\frac{j(1-\alpha^{d-2})}{n(1-\exp(-\lambda_i))(1-\alpha^{d-1})}.
	\end{eqnarray}
Combining~(\ref{eqCompVar21}) and~(\ref{eqCompVar22}), we obtain the assertion.
\qed\end{proof}

Thus,
	\begin{eqnarray*}
	nS_1&\sim&\sum_{i=1}^L\pr\brk{v\in S\wedge N_v=i}\\
		&&\qquad\times\sum_{j=1}^L\brk{((d-1)c-1)ij+i}\pr\brk{w\in S\wedge N_w=j}+
		\pr\brk{w\not\in S\wedge N_w=j}\brk{\gamma_1j+\gamma_2ij}\\
		&=&((d-1)c-1)\frac{R^2}{r}+R+\gamma_2\bar R^2+\sum_{i=1}^N\frac{1-\alpha^{d-2}}{1-\alpha^{d-1}}\pr\brk{N_v=i}\bar R\\
		&=&((d-1)c-1)\frac{R^2}{r}+R+(d-1)(1-\alpha^{d-2})\eps c\bar R^2+\frac{1-\alpha^{d-2}}{1-\alpha^{d-1}}\bar R.
	\end{eqnarray*}
Hence, letting $\Gamma$ be as defined by \eqref{eq:Gamma}
we have $\pr\brk{v,w\in S}-\pr\brk{v\in S}\pr\brk{w\in S})\sim\Gamma/n$.
Consequently,
	$\Var(S)\sim\alpha\Gamma n+\alpha^2r(1-r)n.$

\subsection{Proof of \Lem~\ref{Lemma_WWW}}\label{Sec_WWW}

The probability that a vertex $v\in V\setminus G$ is isolated in $H_{3,G}$ is at least
$(1-p)^{\bink{n_1-1}{d-1}}(1-p_2)^{\bink{n}{d-1}}\sim\exp(-p\bink{n_1-1}{d-1}-\eps p\bink{n}{d-1})\geq\exp(-c)$.
Therefore,
	\begin{equation}\label{eqErwWWW}
	\Erw(|\cW|)\geq(1-o(1))\exp(-c)(n-n_1).
	\end{equation}
To show that $|\cW|$ is concentrated about its mean, we employ the following version of Azuma's inequality (cf.~\cite[p.~38]{JLR00}).

\begin{lemma}\label{Lemma_Azuma}
Let $\Omega=\prod_{i=1}^K\Omega_i$ be a product of probability spaces.
Moreover, let $X:\Omega\rightarrow\RR$ be a random variable that satisfies the following Lipschitz condition.
	\begin{equation}\label{eqLipschitzA}
	\parbox{13cm}{If two tuples $\omega=(\omega_i)_{1\leq i\leq K},\omega'=(\omega_i')_{1\leq i\leq K}\in\Omega$ differ only in their $j$'th components
	for some $1\leq j\leq K$, then $|X(\omega)-X(\omega')|\leq1$.}
	\end{equation}
Then $\pr\brk{|X-\Erw(X)|\geq t}\leq2\exp(-\frac{t^2}{2K})$, provided that $\Erw(X)$ exists.
\end{lemma}
Using \Lem~\ref{Lemma_Azuma}, we shall establish the following.

\begin{corollary}\label{Cor_ConcWWW}
Let $Y$ be a random variable that maps the set of all $d$-uniform hypergraphs with vertex set $V$ to $\brk{0,n}$.
Assume that $Y$ satisfies the following condition.
	\begin{equation}\label{eqLipschitzB}
	\parbox{13cm}{Let $H$ be a hypergraph, and let $e\in\cE(V)$.
	Then $|Y(H)-Y(H+e)|,|Y(H)-Y(H-e)|\leq 1$.}
	\end{equation}
Then $\pr\brk{|Y(H_{3,G})-\Erw(Y(H_{3,G}))|\geq n^{0.66}}\leq\exp(-n^{0.01})$.
\end{corollary}
\begin{proof}
In order to apply \Lem~\ref{Lemma_Azuma}, we need to decompose $H_{3,G}$ into a product $\prod_{i=1}^K\Omega_i$ of probability spaces.
To this end, consider an arbitrary decomposition of the set $\cE(V)$ of all possible edges into sets $\cE_1\cup\cdots\cup\cE_K$
so that $K\leq n$ and $\Erw(E(H_{3,G})\cap\cE_j)\leq n^{0.1}$ for all $1\leq j\leq K$;
such a decomposition exists, because the expected number of edges of $H_{3,G}$ is $\leq\bink{n}dp=O(n)$.
Now, let $\Omega_e$ be a Bernoulli experiment with success probability $p$ for each $e\in\cE(V\setminus G)$,
resp.\ with success probability $p_2$ for $e\in\cE(G,V\setminus G)$.
Then setting $\Omega_i=\prod_{e\in\cE_i}\Omega_e$, we obtain a product decomposition $H_{3,G}=\prod_{i=1}^K\Omega_i$.

In addition, construct for each hypergraph $H$ with vertex set $V$ another hypergraph $H^*$ by removing from $H$ all edges
$e\in\cE_i$ such that $|E(H)\cap\cE_i|\geq 4n^{0.1}$ ($1\leq i\leq K$).
Since $|E(H_{3,G})\cap\cE_i|$ is the sum of two binomially distributed variables, the Chernoff bound~\eqref{eqChernoff} implies that
	$\pr\brk{|E(H_{3,G})\cap\cE_i}|\geq 4n^{0.1})\leq\exp(-n^{0.05})$.
As $K\leq n$, this entails
	\begin{eqnarray}\label{eqAzuma1}
	\pr\brk{H_{3,G}\not=H_{3,G}^*}&\leq&K\exp(-n^{0.05})\leq\exp(-n^{0.04})\mbox{, so that}\\
		|\Erw(Y(H_{3,G}))-\Erw(Y(H_{3,G}^*))|&\leq&1\qquad\mbox{[because $0\leq Y\leq n$].}
		\label{eqAzuma2}
	\end{eqnarray}

As a next step, we claim that $Y^*(H)=\frac14n^{-0.1}Y(H^*)$ satisfies the Lipschitz condition~(\ref{eqLipschitzA}).
For by construction modifying (i.e., adding or removing)  an arbitrary number of edges belonging to a single factor $\cE_i$  can affect
at most $4n^{0.1}$ edges of $H^*$.
Hence,  (\ref{eqLipschitzB}) implies that $Y^*(H)$ satisfies~(\ref{eqLipschitzA}).
Therefore, \Lem~\ref{Lemma_Azuma} entails that
	\begin{equation}\label{eqAzuma3}
	\pr\brk{|Y(H_{3,G}^*)-\Erw(Y(H_{3,G}^*))|\geq n^{0.63}}\leq\pr\brk{|Y^*(H_{3,G})-\Erw(Y^*(H_{3,G}))|\geq n^{0.52}}\leq\exp(-n^{0.02}).
	\end{equation}
Finally, combining~(\ref{eqAzuma1}), (\ref{eqAzuma2}), and~(\ref{eqAzuma3}), we conclude that
	\begin{eqnarray*}
	\pr\brk{|Y(H_{3,G})-\Erw(Y(H_{3,G}))|\geq n^{0.64}}
		&\leq&\pr\brk{|Y^*(H)-\Erw(Y^*(H))|\geq n^{0.63}}+\pr\brk{H_{3,G}\not=H_{3,G}^*}\\
		&\leq&\exp(-n^{0.01}),
	\end{eqnarray*}
thereby completing the proof.
\qed\end{proof}
Finally, since $|\cW|/d$ satisfies (\ref{eqLipschitzB}), \Lem~\ref{Lemma_WWW} follows from
\Cor~\ref{Cor_ConcWWW} and (\ref{eqErwWWW}).

\section{Normality via Stein's Method}\label{Sec_Stein}

In this section we will use Stein's Method to prove that $\order(\hnp)$
as well as $\cS_G$ tend (after suitable normalization) in distribution to the normal distribution.
This proofs Proposition~\ref{Prop_NCLT} as well as Theorem~\ref{Thm_NCLT}
and Lemma~\ref{Lemma_SSSGGGNormal}. First we will define a general setting for using
Stein's Method with random hypergraphs which defines some conditions the random
variables have to fulfill. Then we show in two lemmas (Lemma~\ref{Lemma_SteinCompAux}
and Lemma~\ref{Lemma_SteinAttachAux}) that the random variables corresponding 
to $\order(\hnp)$ and $\cS_G$ do indeed comply to the conditions and last but not least
a quite technical part will show how to derive the limiting distribution from the
conditions.

\subsection{Stein's method for random hypergraphs}

Let $\cE$ be the set of all subsets of size $d$ of $V=\{1,\ldots,n\}$, and let $\cH$ be the power set of $\cE$.
Moreover, let $0\leq p_e\leq 1$ for each $e\in\cE$,
and define a probability distribution on $\cH$ by letting
	$\pr\brk{H}=\prod_{e\in H}p_e\cdot\prod_{e\in\cE\setminus H}1-p_e.$
That is $H\in\cH$ can be considered a random hypergraph with "individual" edge probabilities.

Furthermore, let $\cA$ be a family of subsets of $V$, and let $(\Ya)_{\alpha\in\cA}$ be a family
of random variables. Remember that for $Q\subset V$ we set
$\cE(Q)=\{e\in\cE:e\cap Q\not=\emptyset\}$.
We say that $\Ya$ is \emph{feasible} if the following holds.
\begin{quote}
For any two elements $H,H'\in\cH$ such that $H\cap\cE(\alpha)=H'\cap\cE(\alpha)$
we have $\Ya(H)=\Ya(H')$.
\end{quote}
That means $\Ya$ is feasible if its value depends only on edges having at least one endpoint in $\alpha$.
In addition, set $Y_{\alpha}^S(H)=\Ya(H\setminus\cE(S))$ ($H\in\cH$, $\alpha\in\cA$, $S\subset V$, $S\cap\alpha=\emptyset$).
Thus $Y_{\alpha}^S(H)$ is the value of $\Ya$ after removing all edges incident with $S$.
We define
\begin{eqnarray}
Y&=&\sumall{\alpha}\Ya,\quad\mua=\ex{\Ya},\quad\sigma^2=\var{Y},\quad\Xa=(\Ya-\mua)/\sigma\label{eq:defXa}\\
\Za&=&\sum_{\beta\in\cA}\Zab,\quad\mbox{where }\label{eq:defZa}
\Zab=\sigma^{-1}\times\left\{\begin{array}{cl}
		\Yb&\mbox{ if }\alpha\cap\beta\not=\emptyset,\\
		\Yb-\Yb^\alpha&\mbox{ if }\alpha\cap\beta=\emptyset,
		\end{array}\right.\\
	\label{eq:defVab}
\Vab&=&\sum_{\gamma:\beta\cap\gamma\ne\emptyset\atop\wedge\alpha\cap\gamma=\emptyset}\Yc^\alpha/\sigma +
	\sum_{\gamma:\beta\cap\gamma=\emptyset\atop\wedge\alpha\cap\gamma=\emptyset}
		(\Yc^\alpha-\Yc^{\alpha\cup\beta})/\sigma,\qquad\mbox{and}\\
 \delta&=&\sumall{\alpha}\ex{\abs{\Xa}\Za^2} + 
  \sumall{\alpha,\beta}\left(\ex{\abs{\Xa\Zab\Vab}} + \ex{\abs{\Xa\Zab}}\ex{\abs{\Za+\Vab}}\right).\label{eq:defdelta}
\end{eqnarray}

The following theorem was proven for graphs (i.e. $d=2$) in \cite{BKR}. The argument used
there carries over to the case of hypergraphs without essential modifications. Thus for 
the sake of brevity we omit a detailed proof of this result.
\begin{theorem}\label{thm:bkr}
Suppose that all $\Ya$ are feasible.
If $\delta=o(1)$ as $n\rightarrow\infty$,
then $\frac{Y-\ex{Y}}{\sigma}$ converges to the standard normal distribution.
\end{theorem}
Now the following lemma states that a number of conditions on the expectations of
the product of up to three random variables $\Ya^S$ will suffice for $\delta=o(1)$.
The conditions are identical for both statements we want to prove and we will prove that they
are fulfilled in both cases in the next two sections while the proof of the lemma itself
is deferred to the end of the section.
\begin{lemma}\label{lem:excalc}
Let $k=\tOl$ and let $(\Ya)_{\alpha\in\cA}$ be a feasible family such that $0\leq\Ya\leq k$ for all $\alpha\in\cA$.
If the following six conditions are satisfied, then $\delta=o(1)$ as $n\rightarrow\infty$.
\begin{description}
\item[Y1.] We have $\Erw(Y),\Var(Y)=\Theta(n)$, and
	$
	\sum_{\beta\in\cA:\beta\cap\alpha\not=\emptyset}\mub=\tO{\Erw(Y)/n}=\tOl.
	$
 for any $\alpha\in\cA$
\item[Y2.] Let $\alpha,\beta,\gamma$ be distinct elements of $\cA$.
	Then
	\begin{eqnarray}
	\Ya(\Yb-\Yb^\alpha)\Yb^\alpha&=&0
			\qquad\mbox{if }\alpha\cap\beta=\emptyset\label{eq:del},\\
	\Ya\Yb &=& 0 \qquad\mbox{if }\alpha\cap\beta\not=\emptyset\label{eq:abc},\\
	(Y_{\beta}-Y_{\beta}^{\alpha})Y_{\gamma}^{\alpha}
		=(Y_{\beta}-Y_{\beta}^{\alpha})Y_{\gamma}&=&0
		\qquad\mbox{if $\alpha\cap\beta=\alpha\cap\gamma=\emptyset\not=\beta\cap\gamma$.}
		\label{eqY6}
		\end{eqnarray}
\item[Y3.] For all $\alpha,\beta$ we have
	$\sum_{\gamma:\gamma\cap\beta\not=\emptyset,\,\gamma\cap\alpha=\emptyset}\Erw(Y_{\beta}Y_{\gamma}^{\alpha})\leq k^2\mub.$
\item[Y4.] If $\alpha,\beta\in\cA$ are disjoint, then
	\begin{eqnarray}
	\ex{\Ya\Yb} &=& \tO{\mua\mub},\label{eq:ex5}\\
	\ex{|\Yb-\Yb^\alpha|} &=& \tO{\frac{\mub}{n}},\label{eq:ex6}\\
	\ex{\Ya|\Yb-\Yb^\alpha|} &=& \tO{\frac{\mua\mub}{n}}\label{eq:ex7}.
	\end{eqnarray}
\item[Y5.] If $\alpha,\beta,\gamma\in\cA$ are pairwise disjoint, then
	\begin{eqnarray}
	\ex{\Yb|\Yc^\alpha-\Yc^{\alpha\cup\beta}|} &=& \tO{\frac{\mub\muc}{n}}, \label{eq:ex1}\\
	\ex{|\Yb-\Yb^\alpha|\cdot|\Yc^\alpha-\Yc^{\alpha\cup\beta}|} &=& \tO{\frac{\mub\muc}{n^2}},\label{eq:ex4}\\
	\ex{\Ya|\Yb-\Yb^\alpha|\cdot|\Yc^\alpha-\Yc^{\alpha\cup\beta}|} &=& \tO{\frac{\mua\mub\muc}{n^2}},\label{eq:ex3}\\
	\ex{\Ya|\Yb-\Yb^\alpha|\cdot|\Yc-\Yc^\alpha|} &=& \tO{\frac{\mua\mub\muc}{n^2}},\label{eq:ex0}\\	
	\ex{|(\Yb-\Yb^\alpha)(\Yc-\Yc^\alpha)|}&=&\tO{\frac{\mua\mub}{n^2}}.\label{eq:exNeuACO2}
	\end{eqnarray}
\item[Y6.] If $\alpha,\beta,\gamma\in\cA$ satisfy $\alpha\cap\beta=\alpha\cap\gamma=\emptyset$, then
	\begin{eqnarray}
	\ex{|\Ya^{\beta}-\Ya^{\beta\cup\gamma}|}&=&\tO{\frac{\muc}{n}}.\label{eq:exNeuACO}
	\end{eqnarray}
\end{description}
\end{lemma}

\subsection{Conditions for the normality of $\order(\hnp)$}

In this section we will prove the properties \textbf{Y1}--\textbf{Y6} defined in Lemma~\ref{lem:excalc}
for the case of the normality of $\order(\hnp)$.

Let $k=\tOl$ and let $\cA=\{\alpha\subset V:1\leq|\alpha|\leq k\}$.
Moreover, for $A\subs V$ with $A\cap\alpha=\emptyset$ let $I_{\alpha}^A=1$ if $\alpha$ is a component of $H\setminus\cE(A)$, and $0$ otherwise.
Further, set $Y_{\alpha}^A=|\alpha|\cdot I_{\alpha}^A$.
We briefly write $I_{\alpha}=I_{\alpha}^{\emptyset}$ and $Y_{\alpha}=Y_{\alpha}^{\emptyset}$.
Then $(Y_{\alpha})_{\alpha\in\cA}$ is a feasible family, because whether $\alpha$ is a component or not
only depends on the presence of edges that contain at least one vertex of $\alpha$.

Let $\comp(S)$ denote the even that the subhypergraph of $H$ induced on $S\subset V$ is connected.
If $I_{\alpha}=1$, then $\comp(\alpha)$ occurs.
Moreover, $H$ contains no edges joining $\alpha$ and $V\setminus\alpha$, i.e.,
$H\cap\cE(\alpha, V\setminus\alpha)=\emptyset$.
Since each edge occurs in $H$ with probability $p$ independently, we thus obtain
	\begin{equation}\label{eqACOSteinCompI}
	\pr\brk{I_{\alpha}=1}=\pr\brk{\comp(\alpha)}(1-p)^{|\cE(\alpha, V\setminus\alpha)|}.
	\end{equation}
Furthermore, observe that 
	\begin{equation}\label{eqACOCompImp}
	\forall\alpha\in\cA,\,A\subset B\subset V \setminus\alpha:
		I_{\alpha}^A=1\rightarrow I_{\alpha}^B=1.
	\end{equation}

\textbf{Proof of Y1:}
We know from Theorem~\ref{Thm_NCLT} that $\ex{Y} = \Theta(n)$ and $\var{Y} = \Theta(n)$.
To see that
	$$\sum_{\beta\in\cA:\beta\cap\alpha\not=\emptyset}\mub=\tO{\ex{Y}/n},$$
note that
$\mub:=\ex{\Yb}$ depends only on the size of $\beta$. Thus with $\mu_b=\mub$ for an
arbitrary set $\beta$ of size $b$ we have
$
\ex{Y}=\sumall{\beta}\mub = \sum_{b=1}^k\sum_{\beta\in\cA\atop|\beta|=b}\mub = \sum_{b=1}^k\bink{n}{b}\mu_b
$
while
$
\sum_{\beta\in\cA:\beta\cap\alpha\not=\emptyset}\mub = \sum_{b=1}^k\sum_{\beta\cap\alpha\ne\emptyset\atop|\beta|=b}\mub
\le \sum_{b=1}^k k\bink{n}{b-1}\mu_b.
$

\textbf{Proof of Y2:}
(\ref{eq:del}):
Suppose that $I_{\alpha}=1$.
Then $H$ features no edge that contains a vertex in $\alpha$ and a vertex in $\beta$.
If in addition $I_{\beta}^{\alpha}=1$, then we obtain that $I_{\beta}=1$ as well.
Hence, $Y_{\beta}=Y_{\beta}^{\alpha}$.

(\ref{eq:abc}):
This just means that any two components of $H$ are either disjoint or equal.

(\ref{eqY6}):
To show that $Y_{\gamma}(Y_{\beta}-Y_{\beta}^{\alpha})=0$, assume that $I_{\gamma}=1$.
Then $\gamma$ is a component of $H$, so that $\beta$ cannot be a component, because
$\gamma\not=\beta$ but $\gamma\cap\beta\not=\emptyset$;
hence, $I_{\beta}=0$.
Furthermore, if $\gamma$ is a component of $H$, then $\gamma$ is also a component of $H\setminus\cE(\alpha)$,
so that $I_{\gamma}^{\alpha}=1$.
Consequently, $I_{\beta}^{\alpha}=0$.
Thus, $Y_{\beta}=Y_{\beta}^{\alpha}=0$.

In order to prove that $Y_{\gamma}^{\alpha}(Y_{\beta}-Y_{\beta}^{\alpha})=0$, suppose that $I_{\gamma}^{\alpha}=1$.
Then $I_{\beta}^{\alpha}=0$, because the intersecting sets $\beta,\gamma$ cannot both be components
of $H\setminus\cE(\alpha)$.
Therefore, we also have $I_{\beta}=0$;
for if $\beta$ were a component of $H$, then $\beta$ would also be a component of $H\setminus\cE(\alpha)$.
Hence, also in this case we obtain $Y_{\beta}=Y_{\beta}^{\alpha}=0$.

\textbf{Proof of Y3:}
Suppose that $I_{\beta}=1$, i.e., $\beta$ is a component of $H$.
Then removing the edges $\cE(\alpha)$ from $H$ may cause $\beta$ to split into several components
$B_1,\ldots,B_l$.
Thus, if $Y_{\gamma}^{\beta}>0$ for some $\gamma\in\cA$ such that $\gamma\cap\beta\not=\emptyset$,
then $\gamma$ is one of the components $B_1,\ldots,B_l$.
Since $l\leq|\beta|\leq k$, this implies that given $I_{\beta}=1$ we have the bound
	$\sum_{\gamma:\gamma\cap\beta\not=\emptyset,\,\gamma\cap\alpha=\emptyset}Y_{\gamma}^{\alpha}
		\leq k^2.$
Hence, we obtain \textbf{Y3}.

The following lemma which gives a description of the limited dependence between the random
variables $I_\alpha$ and $I_\beta$ for disjoint $\alpha$ and $\beta$ together
with the fact that $\pr\brk{I_\alpha=1}=O(\mua)$ implies \textbf{Y4}--\textbf{Y6}.

\begin{lemma}\label{Lemma_SteinCompAux}
Let $0\leq l,r\leq2$, and let $\alpha_1,\ldots,\alpha_l,\beta_1,\ldots,\beta_r\in\cA$ be pairwise disjoint.
Moreover, let $A_1,\ldots,A_r,B_1,\ldots,B_r\subset V$ be sets
such that $A_i\subset B_i\subset V\setminus\beta_i$ and $|B_i|\leq2k$ for all $1\leq i\leq r$,
and assume that $\bigcap_{i=1}^rB_i\setminus A_i=\emptyset$.
Then
	$$
	\pr\brk{\bigwedge_{i=1}^l,\,\bigwedge_{j=1}^r\,I_{\alpha_i}=1\wedge I_{\beta_j}^{A_j}\not=I_{\beta_j}^{B_j}}
		\leq\tO{n^{-r}}\prod_{j=1}^l\pr\brk{I_{\alpha_i}=1}\prod_{j=1}^r\pr\brk{I_{\beta_j}=1}.
	$$
\end{lemma}

\begin{proof}
Since~(\ref{eqACOCompImp}) entails that
	$I_{\beta_j}^{A_j}\not=I_{\beta_j}^{B_j}\leftrightarrow I_{\beta_j}^{B_j}=1\wedge I_{\beta_j}^{A_j}=0,$
we have
	\begin{equation}\label{eqACOSteinCompII}
	\pr\brk{\forall i,j:I_{\alpha_i}=1\wedge I_{\beta_j}^{A_j}\not=I_{\beta_j}^{B_j}}
		=\pr\brk{\forall i,j:I_{\alpha_i}=1\wedge I_{\beta_j}^{A_j}=0\wedge I_{\beta_j}^{B_j}=1}.
	\end{equation}
We shall bound the probability on the right hand side in terms of mutually independent events.

If $I_{\alpha_i}=1$ and $I_{\beta_j}^{B_j}=1$ for all $i,j$, then the hypergraphs induced on
$\alpha_i$ and $\beta_j$ are connected, i.e., the events $\comp(\alpha_i)$ and $\comp(\beta_j)$ occur.
Note that these events are mutually independent, because $\comp(\alpha_i)$ (resp.\ $\comp(\beta_j)$) \emph{only}
depends on the presence of edges $e\in\cE(\alpha_i)\setminus\cE(V\setminus\alpha_i)$ (resp.\ $e\subset\cE(\beta_j)\setminus\cE(V\setminus\beta_j)$).

Furthermore, if $\alpha_i$ is a component, then  in $H$ there occur no edges joining $\alpha_i$ and $V\setminus\alpha_i$;
in other words, $H\cap\cE(\alpha_i,V\setminus\alpha_i)=\emptyset$.
However, these events are not necessarily independent, because $\cE(\alpha_1,V\setminus\alpha_1)$
may contain edges that are incident with vertices in $\alpha_2$.
Therefore, we consider the sets
	\begin{eqnarray*}
	\cF(\alpha_i)=\bigcup_{i'\not=i}\alpha_{i'}\cup\bigcup_{j=1}^r\beta_j\cup B_j,&&
	\cD(\alpha_i)=\cE(\alpha_i,V\setminus\alpha_i)\setminus\cE(\cF(\alpha_i)),\\
	\cF(\beta_j)=\bigcup_{i=1}^l\alpha_i\cup\bigcup_{j'\not=j}\beta_{j'}\cup\bigcup_{j'=1}^rB_{j'},&&
	\cD(\beta_j)=\cE(\beta_j,V\setminus\beta_j)\setminus\cE(\cF(\beta_j)).
	\end{eqnarray*}
Then $I_{\alpha_i}=1$ (resp.\ $I_{\beta_j}^{B_j}=1$) implies that $\cD(\alpha_i)\cap H=\emptyset$
(resp.\ $\cD(\beta_j)\cap H=\emptyset$).
Moreover, since the sets $\cD(\alpha_i)$ and $\cD(\beta_j)$ are pairwise disjoint, the events 
$\cD(\alpha_i)\cap H=\emptyset$, $\cD(\beta_i)\cap H=\emptyset$ are mutually independent.

Finally, we need to express the fact that $I_{\beta_j}^{A_j}=0$ but $I_{\beta_j}^{B_j}=1$.
If this event occurs, then $H$ contains an edge connecting $\beta_j$ with $B_j\setminus A_j$,
i.e., $H\cap\cE(\beta_j,B_j\setminus A_j)\not=\emptyset$.
Thus, let $\cQ$ denote the event that $H\cap\cE(\beta_j,B_j\setminus A_j)\not=\emptyset$
for all $1\leq j\leq r$.

Thus, we obtain
	\begin{eqnarray}
	&&\pr\brk{\forall i,j:I_{\alpha_i}=1\wedge I_{\beta_j}^{A_j}=0\wedge I_{\beta_j}^{B_j}=1}\nonumber\\
	&\leq&\pr\brk{\bigwedge_{i=1}^l(\comp(\alpha_i)
				\wedge (\cD(\alpha_i)\cap H=\emptyset))
			\wedge\bigwedge_{j=1}^r(\comp(\beta_j)\wedge(\cD(\beta_j)\cap H=\emptyset))
			\wedge\cQ}\nonumber\\
	&=&\prod_{i=1}^l\pr\brk{\comp(\alpha_i)}\pr\brk{\cD(\alpha_i)\cap H=\emptyset}
		\times\prod_{j=1}^r\pr\brk{\comp(\beta_j)}\pr\brk{\cD(\beta_j)\cap H=\emptyset}
		\times\pr\brk{\cQ}.
	\label{eqACOSteinCompIII}
	\end{eqnarray}
We shall prove below that
	\begin{eqnarray}\label{eqACOSteinCompIV}
	\pr\brk{\cD(\alpha_i)\cap H=\emptyset}&\sim&(1-p)^{|\cE(\alpha_i,V\setminus\alpha_i)|},
		\quad
		\pr\brk{\cD(\beta_j)\cap H=\emptyset}\sim(1-p)^{|\cE(\beta_j,V\setminus\beta_j)|},\\
	\pr\brk{\cQ}&=&\tO{n^{-r}}.
		\label{eqACOSteinCompV}
	\end{eqnarray}
Combining~(\ref{eqACOSteinCompI}) and~(\ref{eqACOSteinCompII})--(\ref{eqACOSteinCompV}), we then obtain the assertion.

To establish~(\ref{eqACOSteinCompIV}), note that by definition
	$\cD(\alpha_i)\subset\cE(\alpha_i,V\setminus\alpha_i)$.
Therefore,
 	\begin{equation}\label{eqACOSteinCompVI}
	\pr\brk{\cD(\alpha_i)\cap H=\emptyset}=
 		(1-p)^{|\cD(\alpha_i)|}\geq(1-p)^{|\cE(\alpha_i,V\setminus\alpha_i)|}.
	\end{equation}
On the other hand, we have $|\alpha_i|,|\cF(\alpha_i)|=\tOl$, and thus
	$|\cE(\alpha_i,\cF(\alpha_i))|\leq|\alpha_i|\cdot|\cF(\alpha_i)|\cdot
		\bink{n}{d-2}=\tO{n^{d-2}}$.
Hence, as $p=O(n^{1-d})$, we obtain
	\begin{eqnarray}\nonumber
	\pr\brk{\cD(\alpha_i)\cap H=\emptyset}&=&(1-p)^{|\cD(\alpha_i)|}
		\leq(1-p)^{|\cE(\alpha_i,V\setminus\alpha_i)|-|\cE(\alpha_i,\cF(\alpha_i))|}\\
			&\sim&(1-p)^{|\cE(\alpha_i,V\setminus\alpha_i)|}
			\exp(p\cdot\tO{n^{d-2}})
			\sim(1-p)^{|\cE(\alpha_i,V\setminus\alpha_i)|}.
	\label{eqACOSteinCompVII}
	\end{eqnarray}
Combining~(\ref{eqACOSteinCompVI}) and~(\ref{eqACOSteinCompVII}), 
we conclude that $\pr\brk{\cD(\alpha_i)\cap H=\emptyset}\sim(1-p)^{|\cE(\alpha_i,V\setminus\alpha_i)|}$.
As the same argument applies to $\pr\brk{\cD(\beta_j)\cap H=\emptyset}$, we thus obtain~(\ref{eqACOSteinCompIV}).

Finally, we prove~(\ref{eqACOSteinCompV}).
If $r=1$, then $H$ contains an edge of $\cE(\beta_1,B_1\setminus A_1)$.
Since
	$$|\cE(\beta_1,B_1\setminus A_1)|\leq|\beta_1|\cdot|B_1\setminus A_1|\cdot n^{d-2}
		=\tO{n^{d-2}},$$
and because each possible edge occurs with probability $p$ independently, the probability of this event
is $\pr\brk{\cQ}\leq\tO{n^{d-2}}p=\tO{n^{-1}}$, as desired.

Now, assume that $r=2$.
Then $H$ features edges $e_j\in \cE(\beta_j,B_j\setminus A_j)$ ($j=1,2$).
\begin{description}
\item[1st case: $e_1=e_2$.]
	In this case, $e_1$ contains a vertex of each of the four sets $\beta_1$, $\beta_2$, $B_1\setminus A_1$,
		$B_2\setminus A_2$.
	Hence, the number of possible such edges is
		$\leq n^{d-4}\prod_{j=1}^2|\beta_j|\cdot|B_j\setminus A_j|=\tO{n^{d-4}}$.
	Consequently, the probability that such an edge occurs in $H$ is $\leq\tO{n^{d-4}}p=\tO{n^{-3}}$.
\item[2nd case: $e_1\not=e_2$.]
	There are $\leq|\beta_j|\cdot|B_j\setminus A_j|\cdot n^{d-2}=\tO{n^{d-2}}$ ways to choose $e_j$ ($j=1,2$).
	Hence, the probability that such edges $e_1,e_2$ occur in $H$ is $\leq\brk{\tO{n^{d-2}}p}^2=\tO{n^{-2}}$.
\end{description}
Thus, in both cases we obtain the bound claimed in~(\ref{eqACOSteinCompV}).
\qed\end{proof}

\subsection{Conditions for the normality of $\cS_G$}

In this section we will prove the properties \textbf{Y1}--\textbf{Y6} defined in Lemma~\ref{lem:excalc}
for the case of the normality of $\cS_G$.

Consider a set $G\subset V$ of size $n_1$.
Let $\cA$ be the set of all subsets $\alpha\subset V\setminus G$ of size $|\alpha|\leq k$.
Moreover, let $p_e=p$ for $e\subset V\setminus G$, $p_e=p_2$ for $e\in\cE(G,V\setminus G)$,
and $p_e=0$ if $e\subset G$.

For $A\subs V$ and $A\cap\alpha=\emptyset$ set $I_{\alpha}^A=1$ if $\alpha$ is a component of $H\setminus\cE(A\cup G)$.
Moreover, let $J_{\alpha}^A=1$ if $(H\setminus\cE(A))\cap\cE(G,\alpha)\ne\emptyset$.
Further, let $K_{\alpha}^A=I_{\alpha}^AJ_{\alpha}^A$ and $Y_{\alpha}^A=|\alpha|K_{\alpha}^A$.
Then
	\begin{equation}\label{eqACOAttach}
	\pr\brk{K_{\alpha}=1}=\Omega(\pr\brk{I_{\alpha}=1}).
	\end{equation}

\textbf{Proof of Y1:}
Using Lemma~\ref{Lemma_ESSSGGG} we have $\ex{Y} = \Theta(n)$ and using Lemma~\ref{Lemma_VarSSSGGG}
we have $\var{Y} = \Theta(n)$.
The proof of the rest of \textbf{Y1} is analogous to the proof of \textbf{Y1} in the case of $\order(\hnp)$.

\textbf{Proof of Y2:}
(\ref{eq:del}):
Suppose that $K_{\alpha}=1$.
Then $I_\alpha=1$, so that $H\setminus\cE(G)$ has no $\alpha$-$\beta$-edges.
Hence, if also $K_{\beta}^{\alpha}=1$, then $\beta$ is a component of $H\setminus\cE(G)$ as well.
Thus, $K_{\beta}=1$, so that $Y_{\beta}=Y_{\beta}^{\alpha}$.

(\ref{eq:abc}):
If $K_{\alpha}=1$, then $\alpha$ is a component of $H\setminus\cE(G)$.
Since any two components of $H\setminus\cE(G)$ are either disjoint or equal, we obtain $I_{\beta}=0$,
so that $Y_{\beta}=0$ as well.

(\ref{eqY6}):
To show that $Y_{\gamma}(Y_{\beta}-Y_{\beta}^{\alpha})=0$, assume that $K_{\gamma}=1$.
Then $I_{\gamma}=1$, i.e., $\gamma$ is a component of $H\setminus\cE(G)$.
Since $\beta\not=\gamma$ but $\beta\cap\gamma\not=\emptyset$,
we conclude that $I_{\beta}=0$.
Furthermore, if $\gamma$ is a component of $H\setminus\cE(G)$, then $\gamma$ is also a component of $H\setminus\cE(G\cup\alpha)$,
whence $I_{\beta}^{\alpha}=0$.
Consequently, $Y_{\beta}=Y_{\beta}^{\alpha}=0$.

In order to prove that $Y_{\gamma}^{\alpha}(Y_{\beta}-Y_{\beta}^{\alpha})=0$, suppose that $K_{\gamma}^{\alpha}=1$.
Then $K_{\gamma}^{\alpha}=1$.
Therefore, $I_{\beta}^{\alpha}=0$, because the intersecting sets $\beta,\gamma$ cannot both be components
of $H\setminus\cE(\alpha)$.
Thus, we also have $I_{\beta}=0$;
for if $\beta$ were a component of $H$, then $\beta$ would also be a component of $H\setminus\cE(\alpha)$.
Hence, also in this case we obtain $Y_{\beta}=Y_{\beta}^{\alpha}=0$.

\textbf{Proof of Y3:}
Suppose that $K_{\beta}=1$.
Then $I_{\beta}=1$, i.e., $\beta$ is a component of $H\setminus\cE(G)$.
Then removing the edges $\cE_{\alpha}$ from $H\setminus\cE(G)$ may cause $\beta$ to split into several components
$B_1,\ldots,B_l$.
Thus, if $Y_{\gamma}^{\beta}>0$ for some $\gamma\in\cA$ such that $\gamma\cap\beta\not=\emptyset$,
then $\gamma$ is one of the components $B_1,\ldots,B_l$.
Since $l\leq|\beta|\leq k$, this implies that given $I_{\beta}=1$ we have the bound
	$$\sum_{\gamma:\gamma\cap\beta\not=\emptyset,\,\gamma\cap\alpha=\emptyset}Y_{\gamma}^{\alpha}
		\leq k^2.$$
Hence, we obtain \textbf{Y3}.

Similar to Lemma~\ref{Lemma_SteinCompAux} the following lemma on the limited dependence
of $K_\alpha$ and $K_\beta$ for disjoint $\alpha$ and $\beta$ implies \textbf{Y4}--\textbf{Y6}.

\begin{lemma}\label{Lemma_SteinAttachAux}
Let $0\leq l,r\leq2$, and let $\alpha_1,\ldots,\alpha_l,\beta_1,\ldots,\beta_r\in\cA$ be pairwise disjoint.
Moreover, let $A_1,\ldots,A_r,B_1,\ldots,B_r\subset V$ be sets
such that $A_i\subset B_i\subset V\setminus\beta_i$ and $|B_i|\leq O(1)$ for all $1\leq i\leq r$,
and assume that $\bigcap_{i=1}^rB_i\setminus A_i=\emptyset$.
Then
	$$\pr\brk{\bigwedge_{i=1}^l\,\bigwedge_{j=1}^r\,
		K_{\alpha_i}=1\wedge K_{\beta_j}^{A_j}\not=K_{\beta_j}^{B_j}}
		\le\tO{n^{-r}}\prod_{j=1}^l\pr\brk{K_{\alpha_i}=1}\prod_{j=1}^r\pr\brk{K_{\beta_j}=1}.$$
\end{lemma}

\begin{proof}
Let $\cP=\pr\brk{\forall i,j:K_{\alpha_i}=1\wedge K_{\beta_j}^{A_j}\not=K_{\beta_j}^{B_j}}$.
If $K_{\beta_j}^{A_j}\not=K_{\beta_j}^{B_j}$, then either $I_{\beta_j}^{A_j}\not=I_{\beta_j}^{B_j}$
or $I_{\beta_j}^{A_j}=I_{\beta_j}^{B_j}=1$ and $J_{\beta_j}^{A_j}\not=J_{\beta_j}^{B_j}$.
Therefore, letting $\cJ=\{j:I_{\beta_j}^{A_j}\not=I_{\beta_j}^{B_j}\}$
and $\bar\cJ=\{1,\ldots,r\}\setminus\cJ$, we obtain
	\begin{equation}\label{eqACOSteinAttachLemI}
	\cP\leq\pr\brk{\bigwedge_{i=1}^lI_{\alpha_i}=1
			\wedge\bigwedge_{j\in\cJ}I_{\beta_j}^{A_j}\not=I_{\beta_j}^{B_j}
			\wedge\bigwedge_{j\in\bar\cJ}\bc{I_{\beta_j}^{A_j}=1\wedge J_{\beta_j}^{A_j}\not=J_{\beta_j}^{B_j}}}.
	\end{equation}
Now, the random variables $I_{\alpha_i}$, $I_{\beta_j}^{A_j}$, and $I_{\beta_j}^{B_j}$ are determined just by
the edges in $\cE\setminus\cE(G)$, while $J_{\beta_j}^{A_j}$ and $J_{\beta_j}^{B_j}$ depend only on the edges
in $\cE(G)$.
Hence, as the edges in $\cE\setminus\cE(G)$ and in $\cE(G)$ occur in $H$ independently, (\ref{eqACOSteinAttachLemI}) yields
	\begin{equation}\label{eqACOSteinAttachLemVI}
	\cP\leq\pr\brk{\bigwedge_{i=1}^lI_{\alpha_i}=1
			\wedge\bigwedge_{j\in\bar\cJ}I_{\beta_j}^{A_j}=1
			\wedge\bigwedge_{j\in\cJ}I_{\beta_j}^{A_j}\not=I_{\beta_j}^{B_j}}
			\cdot\pr\brk{\bigwedge_{j\in\bar\cJ}J_{\beta_j}^{A_j}\not=J_{\beta_j}^{B_j}}.
	\end{equation}
Furthermore, \Lem~\ref{Lemma_SteinCompAux} entails that
	\begin{equation}\label{eqACOSteinAttachLemII}
	\pr\brk{\bigwedge_{i=1}^lI_{\alpha_i}=1
			\wedge\bigwedge_{j\in\bar\cJ}I_{\beta_j}^{A_j}=1
			\wedge\bigwedge_{j\in\cJ}I_{\beta_j}^{A_j}\not=I_{\beta_j}^{B_j}}
			\leq
			\tO{n^{-|\cJ|}}\cdot
				\prod_{i=1}^l\pr\brk{I_{\alpha_i}=1}\cdot\prod_{j=1}^r\pr\brk{I_{\beta_j}=1}.
	\end{equation}
In addition, we shall prove below that
	\begin{equation}\label{eqACOSteinAttachLemIII}
	\pr\brk{\bigwedge_{j\in\bar\cJ}J_{\beta_j}^{A_j}\not=J_{\beta_j}^{B_j}}\leq\tO{n^{-|\bar\cJ|}}.
	\end{equation}
Plugging~(\ref{eqACOSteinAttachLemII}) and~(\ref{eqACOSteinAttachLemIII}) into~(\ref{eqACOSteinAttachLemVI}), we get
	$\cP\leq \tO{n^{-r}}\cdot\prod_{i=1}^l\pr\brk{I_{\alpha_i}=1}\cdot\prod_{j=1}^r\pr\brk{I_{\beta_j}=1},$
so that the assertion follows from~(\ref{eqACOAttach}).

Thus, the remaining task is to establish~(\ref{eqACOSteinAttachLemIII}).
Let us first deal with the case $|\bar\cJ|=1$.
Let $j\in \bar \cJ$.
If $J_{\beta_j}^{A_j}\not=J_{\beta_j}^{B_j}$, then $J_{\beta_j}^{A_j}=1$ and $J_{\beta_j}^{B_j}=0$, because $A_j\subset B_j$.
Thus, $\beta_j$ is connected to $G$ via an edge that is incident with $A_j\setminus B_j$;
that is, $H\cap\cE(\beta_j,B_j\setminus A_j)\not=\emptyset$.
Since there are $|\cE(\beta_j,B_j\setminus A_j)|\leq|\beta_j|\cdot|B_j|\cdot n^{d-2}=\tO{n^{d-2}}$
such edges to choose from, and because each such edge is present with probability $p_2=O(n^{1-d})$,
we conclude that
	$\pr\brk{J_{\beta_j}^{A_j}\not=J_{\beta_j}^{B_j}}
		\leq\pr\brk{H\cap\cE(\beta_j,B_j\setminus A_j)\not=\emptyset}
		\leq\tO{n^{d-2}}p_2=\tO{n^{-1}},$
whence we obtain~(\ref{eqACOSteinAttachLemIII}).

Finally, suppose that $|\bar\cJ|=2$.
If $J_{\beta_j}^{A_j}\not=J_{\beta_j}^{B_j}$ for $j=1,2$, then
there occur edges $e_j\in H\cap\cE(\beta_j,B_j\setminus A_j)$ ($j=1,2$).
\begin{description}
\item[1st case: $e_1=e_2$.]
	In this case $e_1=e_2$ is incident with all four sets $\beta_j,B_j\setminus A_j$ ($j=1,2$).
	Hence, as the number of such edges is
		$\leq n^{d-4}\prod_{j=1}^2|\beta_j|\cdot|B_j\setminus A_j|\leq\tO{n^{d-4}}$
	and each such edge occurs with probability $p_2=O(n^{1-d})$, the probability that the 1st
	case occurs is $\tO{n^{d-4}}p_2=\tO{n^{-3}}$.
\item[2nd case: $e_1\not=e_2$.]
	There are $\leq|\beta_j|\cdot|B_j\setminus A_j|\cdot n^{d-2}\leq\tO{n^{d-2}}$ ways to choose
	$e_j$ for $j=1,2$, each of which is present with probability $p_2=O(n^{1-d})$ independently.
	Hence, the probability that the second case occurs is bounded by $\brk{\tO{n^{d-2}}p_2}^2\leq\tO{n^{-2}}$.
\end{description}
Thus, the bound~(\ref{eqACOSteinAttachLemIII}) holds in both cases.
\qed\end{proof}

\subsection{Proof of \Lem~\ref{lem:excalc}}

All we need to show is that the conditions defined in \Lem~\ref{lem:excalc} imply that
$\delta$ as defined by \eqref{eq:defdelta} tends to 0. We will do so by proving that each
of the three summands contributing to $\delta$ is $\tO{\sigma^{-3}\ex{Y}}$. Together
with condition \textbf{Y1}, stating that $\ex{Y}, \sigma^2 = \Theta(n)$, this implies
the statement. We formulate one lemma for each summand, bounding the expectations
using conditions \textbf{Y1}--\textbf{Y6}. The proof of the lemmas are mainly long and
technical computations then.

\begin{lemma}\label{Lemma_ACOexcalcI}
$\sumall{\alpha}\ex{\abs{\Xa}\Za^2} = \tO{\sigma^{-3}\ex{Y}}$
\end{lemma}
\begin{proof}
Let
\begin{eqnarray*}
S_1=\sumall{\alpha}\ex{\Ya\left(\sumabc\Yb\right)^2},&&
	S_2=\sumall{\alpha}\ex{\mua\left(\sumabc\Yb\right)^2},\\
S_3=\sumall{\alpha}\ex{\Ya\left(\sumabd(\Yb-\Yb^\alpha)\right)^2},&&
	S_4=\sumall{\alpha}\ex{\mua\left(\sumabd(\Yb-\Yb^\alpha)\right)^2}.
\end{eqnarray*}
Since $\Xa=(\Ya-\mua)/\sigma\leq(\Ya+\mua)/\sigma$, (\ref{eq:defXa}) entails that
\begin{eqnarray*}
\ex{\abs{\Xa}\Za^2}
 &\le&
	2\sigma^{-3}\ex{(\Ya+\mua)\left(\left(\sumabc\Yb\right)^2+\left(\sumabd(\Yb-\Yb^\alpha)\right)^2\right)}\\
&\leq&2\sigma^{-3}(S_1+S_2+S_3+S_4).
\end{eqnarray*}
Therefore, it suffices to show that $S_j=\tO{\Erw(Y)}$ for $j=1,2,3,4$.

Regarding $S_1$, we obtain the bound
$$
S_1 = \sumall{\alpha}\sumabc\sumacc\ex{\Ya\Yb\Yc}\;\stacksign{(\ref{eq:abc})}{\leq}\;
		k^2\sum_{\alpha\in\cA}\ex{\Ya}\leq \tO{\ex{Y}}.
$$

With respect to $S_2$, note that due to~(\ref{eq:abc}) and~(\ref{eq:ex5})
we have $\ex{\Yb\Yc}\leq k\mub$ if $\beta=\gamma$,
$\ex{\Yb\Yc}=0$ if $\beta\not=\gamma$ but $\beta\cap\gamma\not=\emptyset$,
and $\ex{\Yb\Yc}=\tO{\mub\muc}$ if $\beta\cap\gamma=\emptyset$.
Consequently,
\begin{eqnarray}
S_2 &=& \sumall{\alpha}\mua\sumabc\sumacc\ex{\Yb\Yc}\nonumber\\
 &\le&\sumall{\alpha}\mua\sumabc\sumacc \tO{\mub\muc}
	\;\stacksign{\bf{Y1}}{\leq}\;\tO{\Erw(Y)}.\label{eqACOS2II}
\end{eqnarray}

Concerning $S_3$, we obtain
\begin{eqnarray*}
S_3
 &=& \sumall{\alpha}\sumabd\sumacd\ex{\Ya(\Yb-\Yb^\alpha)(\Yc-\Yc^\alpha)}\\
 &\stackrel{\eqref{eq:ex0},\,(\ref{eqY6})}{\leq}&
	\sumall{\alpha}\sumabd\sumacd \tO{\mua\mub\muc n^{-2}}\\
 &\le&
	\tO{n^{-2}}
\Erw(Y)^3\;\stacksign{\bf{Y1}}{\leq}\;\tO{\Erw(Y)}.
\end{eqnarray*}

To bound $S_4$, we note that for disjoint $\alpha,\beta\in\cA$ 
and $\gamma\in\cA$ disjoint from $\alpha$
the conditions (\ref{eq:ex6}), (\ref{eq:abc}), and~(\ref{eq:exNeuACO2}) yield
$$
\ex{|(\Yb-\Yb^\alpha)(\Yc-\Yc^\alpha)|} =\left\{
 \begin{array}{cl}
    \tO{\frac{\mub}{n}}       &\mbox{if }\beta=\gamma\\
    0                       &\mbox{if }\beta\ne\gamma,\beta\cap\gamma\ne\emptyset\\
    \tO{\frac{\mub\muc}{n^2}} &\mbox{if }\beta\cap\gamma=\emptyset.
 \end{array}\right.
$$
Therefore,
\begin{eqnarray*}
\sumabd\sumacd\ex{|(\Yb-\Yb^\alpha)(\Yc-\Yc^\alpha)|}
 &\le&\sumall{\beta}\sumall{\gamma}\tO{\frac{\mub\muc}{n^2}}
 + \sumall{\beta}\tO{\frac{\mub}{n}}\\
 &\le&\tO{\Erw(Y)^2/n^2}+\tO{\Erw(Y)/n}\\
 &=&\tOl.
\end{eqnarray*}
Hence, we obtain
$S_4\leq\sumall{\alpha}\mua\sumabc\sumacc\ex{(\Yb-\Yb^\alpha)(\Yc-\Yc^\alpha)} \leq \tO{\ex{Y}}$.
\qed\end{proof}

\begin{lemma}\label{Lemma_ACOexcalcII}
$\sumall{\alpha}\sumall{\beta}\ex{\abs{\Xa\Zab\Vab}} = \tO{\sigma^{-3}\ex{Y}}$
\end{lemma}
\begin{proof}
Let
	$S_1=\sumabc\ex{\abs{\Xa\Yb\Vab}}$ and
		$S_2=\sumabd\ex{\abs{\Xa(\Yb-\Yb^\alpha)\Vab}}$.
Then the definition~(\ref{eq:defZa}) of $\Zab$ yields that
$\sumall{\alpha}\sumall{\beta}\ex{\abs{\Xa\Zab\Vab}}\leq\sigma^{-1}(S_1+S_2)$
Hence, it suffices to show that $S_1,S_2=\tO{\sigma^{-2}\ex{Y}}$.

To bound $S_1$, we note that
$\Ya\Yb=0$ if $\alpha\cap\beta\ne\emptyset$ but $\alpha\not=\beta$ by~(\ref{eq:abc}), and that
$\Vab=0$ if $\alpha=\beta$ by the definition~(\ref{eq:defVab}) of $\Vab$.
Thus, if $\alpha\cap\beta\not=\emptyset$, then
\begin{equation}\label{eqACOLemmaPart2I}
\ex{\abs{\Xa\Yb\Vab}} \stackrel{\eqref{eq:defXa}}{\le} \sigma^{-1}\ex{(\Ya+\mua)\abs{\Yb\Vab}} \leq
	\sigma^{-1}\mua\ex{\abs{\Yb\Vab}}.
\end{equation}
Furthermore,
\begin{eqnarray}\label{eqACOLemmaPart2FixI}
T_1(\alpha,\beta)&=&\sum_{\gamma:\gamma\cap\beta\not=\emptyset,\,\gamma\cap\alpha=\emptyset}\Erw\brk{|Y_{\beta}Y_{\gamma}^{\alpha}}
	\stacksign{Y7}{\leq}k^2\mub.
	\\
T_2(\alpha)&=&\sumabc\sum_{\gamma:\beta\cap\gamma=\emptyset\atop\wedge\alpha\cap\gamma=\emptyset}
	\ex{\Yb|\Yc^\alpha-\Yc^{\alpha\cup\beta}|}
\;\stackrel{\eqref{eq:ex1}}{\leq}\;
\sum_{\beta:\alpha\cap\beta\not=\emptyset}
	\sum_{\gamma:\beta\cap\gamma=\emptyset\atop\wedge\alpha\cap\gamma=\emptyset}
		\tO{\frac{\mub\muc}{n}}\nonumber\\
	&\leq&\tO{n^{-1}}\brk{\sum_{\gamma\in\cA}\mu_\gamma}\sumabc\mub\nonumber\\
	&\stacksign{\bf{Y1}}{\leq}&\tO{n^{-1}\Erw(Y)}=\tOl.
	\label{eqACOLemmaPart2FixII}
\end{eqnarray}
Combining~(\ref{eqACOLemmaPart2I})--(\ref{eqACOLemmaPart2FixII}), we get
\begin{eqnarray*}
S_1&\le&\sigma^{-1}\sum_{\alpha\in\cA}\sumabc\mua\ex{\abs{\Yb\Vab}}
 \;\stacksign{(\ref{eq:defVab})}{\leq}\;
	\sigma^{-2}\sum_{\alpha\in\cA}
		\mua\brk{T_2(\alpha)+\sum_{\beta:\alpha\cap\beta\not=\emptyset}T_1(\alpha,\beta)}\\
	&\leq&\tO{\sigma^{-2}}\brk{\Erw(Y)+k^2\sum_{\beta:\alpha\cap\beta\not=\emptyset}\mub}
	\;\stacksign{\bf{Y1}}{\leq}\;
		\tO{\sigma^{-2}\Erw(Y)}
\end{eqnarray*}
	
To bound $S_2$, let $\alpha,\beta\in\cA$ be disjoint.
As $\Xa\leq(\Ya+\mua)/\sigma$, we obtain
\begin{eqnarray*}\nonumber
\ex{\abs{\Xa(\Yb-\Yb^\alpha)\Vab}}
 &\le&\sigma^{-1}\ex{\abs{(\Ya+\mua)(\Yb-\Yb^\alpha)\Vab}}\\
 &\stackrel{\eqref{eq:defVab},\,(\ref{eqY6})}{\leq}&\sigma^{-2}\ex{\abs{(\Ya+\mua)(\Yb-\Yb^\alpha)\Yb^\alpha}}\nonumber\\
 &&\quad + \sigma^{-2}\sum_{\gamma:\beta\cap\gamma=\emptyset\atop\wedge\alpha\cap\gamma=\emptyset}
		\ex{\abs{(\Ya+\mua)(\Yb-\Yb^\alpha)(\Yc^\alpha-\Yc^{\alpha\cup\beta})}}\nonumber\\
 &\leq&\sigma^{-2}(T_1+T_2+T_3+T_4),
\label{eqACOLemmaPart2II}
\end{eqnarray*}
where
\begin{eqnarray*}
T_1=\ex{\abs{\Ya(\Yb-\Yb^\alpha)\Yb^\alpha}},&&
T_2=\mua\ex{\abs{(\Yb-\Yb^\alpha)\Yb^\alpha}},\\
T_3=\sum_{\gamma:\beta\cap\gamma=\emptyset\atop\wedge\alpha\cap\gamma=\emptyset}\ex{\abs{\Ya(\Yb-\Yb^\alpha)(\Yc^\alpha-\Yc^{\alpha\cup\beta})}},
&&T_4=\mua\sum_{\gamma:\beta\cap\gamma=\emptyset\atop\wedge\alpha\cap\gamma=\emptyset}\ex{\abs{(\Yb-\Yb^\alpha)(\Yc^\alpha-\Yc^{\alpha\cup\beta})}}.
\end{eqnarray*}
Now, $T_1=0$ by~\eqref{eq:del}.
Moreover, bounding $T_2$ by \eqref{eq:ex6},
$T_3$ by  \eqref{eq:ex3} and $T_4$ by \eqref{eq:ex4}, we obtain
\begin{eqnarray*}
\sigma^2\ex{\abs{\Xa(\Yb-\Yb^\alpha)\Vab}}
 &\le&\tO{\frac{\mua\mub}{n}} + \sum_{\gamma:\beta\cap\gamma=\emptyset\atop\wedge\alpha\cap\gamma=\emptyset}
	\tO{\frac{\mua\mub\muc}{n^2}}\\
 &=&\tO{\frac{\mua\mub}{n}}.
\end{eqnarray*}
Thus, (\ref{eqACOLemmaPart2II}) yields
$S_2\leq
 \sigma^{-2}\sumabd \tO{\frac{\mua \mub}{n}}
=\tO{n^{-1}\sigma^{-2}\Erw(Y)^2}
=\tO{\sigma^{-2}\Erw(Y)},$
as desired.
\qed\end{proof}

\begin{lemma}\label{Lemma_ACOexcalcIII}
$\sumall{\alpha}\sumall{\beta}\ex{\abs{\Xa\Zab}}\ex{\abs{\Za+\Vab}} = \tO{\sigma^{-3}\ex{Y}}$
\end{lemma}
\begin{proof}
Since $|\sigma\Xa|\leq\Ya+\mua$,
\begin{eqnarray}\nonumber
\sumall{\alpha}\sumall{\beta}\ex{\abs{\Xa\Zab}}\ex{\abs{\Za+\Vab}}
 &\le &\sigma^{-1} \Big(\sumall{\alpha}\sumall{\beta}\mua\ex{\abs{\Zab}}(\ex{\abs{\Za}}+\ex{\abs{\Vab}}) +\nonumber\\
   &  &\qquad\qquad\qquad\qquad\ex{\Ya\abs{\Zab}}(\ex{\abs{\Za}}+\ex{\abs{\Vab}})\Big).
	\label{eqACOPart3V}
\end{eqnarray}
Furthermore, we have the three estimates
\begin{eqnarray}
\sigma\ex{\abs{\Za}}&\le&\sigma\sumall{\beta}\ex{\abs{\Zab}}\;\stackrel{\eqref{eq:defZa}}{=}\;
	\sumabc\mub + \sumabd\ex{\abs{\Yb-\Yb^\alpha}}\nonumber\\
 	&\stackrel{\eqref{eq:ex6},\,\bf{Y1}}{\leq}&
		\sum_{\beta\in\cA}\tO{n^{-1}\mub}
=\tOl,\label{eqACOPart3I}\\
\sigma\ex{\abs{\Vab}} &\stackrel{\eqref{eq:defVab}}{\le}&
	\sum_{\gamma:\beta\cap\gamma\ne\emptyset\atop\wedge\alpha\cap\gamma=\emptyset}
		\ex{\abs{\Yc^\alpha}}
	+\sum_{\gamma:\beta\cap\gamma=\emptyset\atop\wedge\alpha\cap\gamma=\emptyset}
		\ex{\abs{\Yc^\alpha-\Yc^{\alpha\cup\beta}}}\nonumber\\
 	&\stackrel{\eqref{eq:exNeuACO},\,\bf{Y1}}{=}&
		\sum_{\gamma\in\cA}\tO{n^{-1}\muc}
\leq \tOl,\label{eqACOPart3II}\\
\sumall{\beta}\sigma\ex{\Ya\abs{\Zab}}
 &\stackrel{\eqref{eq:defZa}}{=}&\sumabc\ex{\Ya\Yb} + \sumabd\ex{\Ya\abs{\Yb-\Yb^\alpha}}\nonumber\\
 &\stackrel{\eqref{eq:abc},\,\eqref{eq:ex7}}{=}&
	k\mua + \sumabd\frac{\mua\mub}{n}
	= \tO{\mua}.\label{eqACOPart3III}
\end{eqnarray}
Now, (\ref{eqACOPart3I})--(\ref{eqACOPart3III})
yield
\begin{eqnarray}
\sumall{\alpha}\sumall{\beta}\mua\ex{\abs{\Zab}}(\ex{\abs{\Za}}+\ex{\abs{\Vab}})&=&\tO{\sigma^{-2}}\sumall{\alpha}\mua\nonumber\\
 &=& \tO{\sigma^{-2}\ex{Y}},\label{eqACOPart3VI}\\
\sumall{\alpha}\sumall{\beta}\ex{\Ya\abs{\Zab}}(\ex{\abs{\Za}}+\ex{\abs{\Vab}})&=&\tO{\sigma^{-2}}\sumall{\alpha}\mua\nonumber\\
 &=& \tO{\sigma^{-2}\ex{Y}}.
	\label{eqACOPart3VII}
\end{eqnarray}
Combining~(\ref{eqACOPart3V}), (\ref{eqACOPart3VI}), and~(\ref{eqACOPart3VII}), we obtain the assertion.
\qed\end{proof}
Finally, \Lem~\ref{lem:excalc} is an immediate
consequence of \Lem s~\ref{Lemma_ACOexcalcI}--\ref{Lemma_ACOexcalcIII}.

\section{Conclusion}\label{Sec_Conclusion}

Using a purely probabilistic approach, we have established a local limit theorem for $\order(\hnp)$.
This result has a number of interesting consequences, which we derive in a follow-up paper~\cite{WirSelbst}.
Namely, via Fourier analysis the \emph{univariate} local limit theorem (\Thm~\ref{Thm_Nlocal}) can be transformed into a \emph{bivariate} one
that describes the joint distribution of the order and the number of edges of the largest component.
Furthermore, since given its number of vertices and edges the largest component is
a uniformly distributed connected graph, this bivariate limit theorem yields an asymptotic formula
for the number of connected hypergraphs with a given number of vertices and edges.
Thus, we can solved an involved enumerative problem (``how many connected hypergraphs with $\nu$ vertices and $\mu$ edges
exist?'') via a purely probabilistic approach.

The techniques that we have presented in the present paper appear rather generic and may
apply to further related problems.
For instance, it seems possible to extend our proof of \Thm~\ref{Thm_Nlocal} to the regime $c=\bink{n-1}{d-1}p=(d-1)^{-1}+o(1)$.
In addition, it would be interesting to see whether our techniques can be used to obtain limit theorems
for the $k$-core of a random graph, or for the largest component of a random digraph.


\begin{thebibliography}{99} 


\bibitem{AR05}
Andriamampianina, T., Ravelomanana, V.:
Enumeration of connected uniform hypergraphs.
Proceedings of FPSAC 2005

\bibitem{BBF00} 
Barraez, D., Boucheron, S., Fernandez de la Vega, W.: On the 
fluctuations of the giant component. Combinatorics, Probability 
and Computing {\bf 9} (2000)  287--304 
 
 \bibitem{BKR}
 Barbour, A.D., \Karonski, M., \Rucinski, A.:
 A central limit theorem for decomposable random variables with applications to random graphs.
 J.~Combin.\ Theory Ser.\ B {\bf 47} (1989) 125--145

\bibitem{WirSelbst} 
Behrisch, M., Coja-Oghlan, A., Kang, M.:
Local limit theorems and the number of connected hypergraphs.
Preprint (2007).
 
\bibitem{BCM90} 
Bender, E.A., Canfield, E.R., McKay, B.D.: The asymptotic number 
of labeled connected graphs with a given number of vertices and 
edges. Random Structures and Algorithms {\bf 1} (1990) 127--169 

\bibitem{BB}
Bollob\'as, B.: Random graphs. 2nd edition.
Cambridge University Press (2001)

 
 
\bibitem{CMS04}
Coja-Oghlan, A., Moore, C., Sanwalani, V.:
Counting connected graphs and hypergraphs via the probabilistic method.
To appear in Random Structures and Algorithms.

\bibitem{ER60}
Erd{\H{o}}s, P., R{\'e}nyi, A.:
On the evolution of random graphs.
Publ. Math. Inst. Hung. Acad. Sci. {\bf 5} (1960) 17--61.



\bibitem{HS05}
van der Hofstad, R., Spencer, J.:
Counting connected graphs asymptotically.
To appear in the European Journal on Combinatorics.



\bibitem{JansonMST} 
Janson, S.: The minimal spanning tree in a complete graph and a functional limit theorem for trees in a random graph.
Random Structures and Algorithms {\bf 7} (1995) 337--355

\bibitem{JLR00} 
Janson, S., {\L}uczak, T, Ruci\'nski, A.: Random Graphs, Wiley 
2000 

\bibitem{KL97} 
\Karonski, M., \L uczak, T.: The number of connected sparsely edged uniform hypergraphs. 
Discrete Math.\ {\bf 171} (1997) 153--168 
 
\bibitem{KL02} 
\Karonski, M., \L uczak, T.: The phase transition in a random hypergraph. 
J.\ Comput.\ Appl.\ Math.\ {\bf 142} (2002) 125--135 
 
 


 
 


\bibitem{P90}
Pittel, B.: On tree census and the giant component in sparse random graphs.
Random Structures and Algorithms {\bf1} (1990) 311--342

\bibitem{PW05} 
Pittel, B., Wormald, N.C.: Counting connected graphs inside out. 
J.\ Combin.\ Theory, Series B {\bf93} (2005) 127--172

\bibitem{RR}
Ravelomanana, V., Rijamamy, A.L.:
Creation and growth of components in a random hypergraph process.
Preprint (2005).

\bibitem{SPS85} 
Schmidt-Pruzan, J., Shamir, E.: Component structure in the 
evolution of random hypergraphs. Combinatorica {\bf 5} (1985) 
81--94 

\bibitem{Stein}
Stein, C.:
A bound for the error in the normal approximation to the distribution of a sum of dependent variables.
Proc.\ 6th Berkeley Symposium on Mathematical Statistics and Probability (1970) 583--602

\bibitem{Ste70}
Stepanov, V.~E.:
On the probability of connectedness of a random graph ${\cal G}_m(t)$.
Theory Prob. Appl. {\bf15} (1970) 55--67.
 
\end{thebibliography}
\end{document}